\documentclass[11pt]{article}
\usepackage[margin=1in]{geometry}
\usepackage[utf8]{inputenc}
\usepackage{amsmath,amssymb}
\usepackage[bookmarks, colorlinks=true, plainpages = false, citecolor = blue,linkcolor=blue,urlcolor = blue, filecolor = blue]{hyperref}
\usepackage{amsthm}
\usepackage{dsfont}
\usepackage{dsfont}

\newcommand{\Red}{\mathsf{Red}}
\newcommand{\Risk}{\mathsf{Risk}}
\newcommand{\mem}{\mathsf{mem}}
\newcommand{\risk}{\Risk}
\newcommand{\redd}{\Red}

\newcommand{\red}{\redd}
\renewcommand{\l}{\ell}
\newcommand{\kl}{\operatorname{KL}}
\newcommand{\tv}{\operatorname{TV}}
\newcommand{\E}{\mathbb{E}}

\renewcommand{\P}{\mathbb{P}}

\newcommand{\N}{\mathcal{N}}

\newcommand{\eps}{\varepsilon}
\newcommand{\one}{\mathds{1}}

\usepackage{graphicx,graphics}
\usepackage{enumerate}
\newcommand{\Bern}{\textup{Bern}}

\newcommand{\supp}{\operatorname{supp}}

\newcommand{\T}{\operatorname{C}}
\usepackage{mdframed}
\usepackage{nameref, cleveref}

\usepackage{prettyref}
\newrefformat{eq}{(\ref{#1})}
\newrefformat{chap}{Chapter~\ref{#1}}
\newrefformat{sec}{Section~\ref{#1}}
\newrefformat{alg}{Algorithm~\ref{#1}}
\newrefformat{fig}{Fig.~\ref{#1}}
\newrefformat{tab}{Table~\ref{#1}}
\newrefformat{rmk}{Remark~\ref{#1}}
\newrefformat{clm}{Claim~\ref{#1}}
\newrefformat{def}{Definition~\ref{#1}}
\newrefformat{cor}{Corollary~\ref{#1}}
\newrefformat{lmm}{Lemma~\ref{#1}}
\newrefformat{prop}{Proposition~\ref{#1}}
\newrefformat{app}{Appendix~\ref{#1}}
\newrefformat{appdx}{Appendix~\ref{#1}}
\newrefformat{hyp}{Hypothesis~\ref{#1}}
\newrefformat{thm}{Theorem~\ref{#1}}
\newrefformat{ass}{Assumption~\ref{#1}}
\newrefformat{conj}{Conjecture~\ref{#1}}

\newcommand{\qth}[1]{\left[ #1 \right]}

\newcommand{\calA}{{\mathcal{A}}}

\newcommand{\calP}{{\mathcal{P}}}
\newcommand{\calQ}{{\mathcal{Q}}}

\newcommand{\calX}{{\mathcal{X}}}

\newcommand{\naturals}{\mathbb{N}}
\newcommand{\integers}{\mathbb{Z}}

\newcommand{\Expect}{\mathbb{E}}

\newcommand{\PHMM}{\calP^{\sf HMM}}

\newcommand{\RiskHMM}{\Risk_{\sf HMM}}
\newcommand{\Prenew}{\calP^{\sf rnwl}}

\newcommand{\Riskrenew}{\Risk_{\sf rnwl}}

\newcommand{\stepa}[1]{\overset{\rm (a)}{#1}}
\newcommand{\stepb}[1]{\overset{\rm (b)}{#1}}
\newcommand{\stepc}[1]{\overset{\rm (c)}{#1}}
\newcommand{\stepd}[1]{\overset{\rm (d)}{#1}}
\newcommand{\stepe}[1]{\overset{\rm (e)}{#1}}

\newtheorem{lemma}{Lemma}
\newtheorem{theorem}{Theorem}
\newtheorem{corollary}{Corollary}

\newtheorem{proposition}{Proposition}
\newtheorem{remark}{Remark}
\usepackage[backend=biber, style=alphabetic,]{biblatex}
\newtheorem{conjecture}{Conjecture}
\title{Prediction from compression for models with infinite memory,\\ with applications to hidden Markov and renewal processes}
\author{Yanjun Han, Tianze Jiang, Yihong Wu\footnote{ Yanjun Han is with the Courant Institute of Mathematical Sciences and the Center for Data Science, New York University. Email: \url{yanjunhan@nyu.edu}. 
Yanjun Han was generously supported by the Norbert Wiener postdoctoral fellowship in statistics at MIT IDSS and a startup grant at New York University. Tianze Jiang is with MIT EECS. Email: \url{tjiang@mit.edu}. Yihong Wu is with the Department of Statistics and Data Science, Yale University. Email: \url{yihong.wu@yale.edu}.} }
\begin{document}
\maketitle
\begin{abstract}
    Consider the problem of predicting the next symbol given a sample path of length $n$, whose joint distribution belongs to a distribution class that may have long-term memory. The goal is to compete with the conditional predictor that knows the true model. For both hidden Markov models (HMMs) and renewal processes, we determine the optimal prediction risk in Kullback-Leibler divergence up to universal constant factors. 
Extending existing results in finite-order Markov models \cite{HJW23} and drawing ideas from universal compression, 
the proposed estimator has a prediction risk bounded by redundancy of the distribution class and a memory term that accounts for the long-range dependency of the model. Notably, for HMMs with bounded state and observation spaces, a polynomial-time estimator based on dynamic programming is shown to achieve the optimal prediction risk $\Theta(\frac{\log n}{n})$; prior to this work, the only known result of this type is $O(\frac{1}{\log n})$ obtained using Markov approximation \cite{sharan2018prediction}. Matching minimax lower bounds are obtained by making connections to redundancy and mutual information via a reduction argument.

\end{abstract}

\newpage
\tableofcontents
\newpage

\section{Introduction}
Consider the following ``ChatGPT'' style of problem: Observing a sample path $X^n\triangleq(X_1,\ldots,X_n)$ of a random process, one is tasked to predict the next (unseen) symbol $X_{n+1}$. Mathematically, this boils down to estimating the \emph{conditional} distribution $P_{X_{n+1}|X_1^n}$, which informs downstream tasks such as finding the top few most likely realizations in autocomplete or text generation in language models. This is a well-defined but non-standard statistical problem, in that the quantity to be estimated is random and data-dependent, unless the data are i.i.d., in which case the problem is nothing but density estimation and the optimal rate under, say, Kullback-Leibler (KL) divergence loss, is the classical ``parametric rate'' $\frac{k}{n}$ achieved by smoothed empirical distribution, where $k$ and $n$ refers to the alphabet and sample size respectively. As such, the first non-trivial instance is Markov model and of interest to applications such as natural language processing are large state spaces.

The study of this problem was initiated by \cite{Falahatgar16learning} focusing on two-state Markov chains, who showed, via a tour-de-force argument, the surprising result that the optimal KL prediction risk is $\Theta(\frac{\log\log n}{n})$, strictly slower than the parametric rate. Their ad hoc techniques are difficult to extend to larger state space, unless extra conditions are assumed such as a large spectral gap \cite{HOP18}.
Although such mixing conditions are necessary for parameter estimation, they are not for prediction. Indeed, a chain that moves at a glacial speed is in fact easy to predict but estimating the transition probabilities is impossible.
This is a significant conceptual distinction between estimation and prediction, the latter of which can be studied meaningfully assumption-free without even identifiability conditions.

Departing from conventional approaches based on concentration inequalities of Markov chains which inevitably involves mixing conditions, a strategy based on \textit{universal compression} is proposed in \cite{han2021optimal,HJW23} for prediction of Markov chains. They showed, by means of information-theoretic arguments, that the optimal prediction risk is within universal constant factors of the so-called \textit{redundancy}, a central quantity in universal compression that measures the KL radius of the model class. Furthermore, this reduction is also \textit{algorithmic}: if there is a computationally efficient probability assignment that achieves the redundancy, one can construct an efficient predictor with guaranteed optimality. However, their method is limited to Markov models with a finite order.


The main goal of this work is to extend these techniques based on universal compression beyond models with finite memory to those with infinite memory, in particular, \textit{hidden Markov models} (HMMs) and \textit{renewal processes}.
Along the way, we obtain new theoretical and computational results for prediction HMM that improve the state of the art.

\subsection{Main results}
\label{sec:main}
Let us begin with the formulation of the \textit{prediction risk} for a general model class.
For $n\in\naturals\triangleq \{1,2,\ldots\}$, let $\calP_{n+1}$ be a collection of joint distributions $P_{X^{n+1}}$ for $X^{n+1}\triangleq(X_1,\ldots,X_{n+1})$, where each observation $X_t$ takes values in some space $\calX$.
The prediction risk of the next unseen symbol $X_{n+1}$ based on the trajectory $X_1,\ldots,X_n$ is the average $\kl$ risk of estimating the (random, data-dependent) distribution $P_{X_{n+1}|X^n}$. Any such estimator can be written as a conditional distribution $Q_{X_{n+1}|X^{n}}$, whose worst-case prediction risk over the model class is 
\begin{equation}
    \Risk(Q_{X_{n+1}|X^{n}}; \calP_{n+1}) \triangleq \sup_{P_{X^{n+1}}\in\calP_{n+1}} \Expect_{X^{n+1}\sim P_{X^{n+1}}} \qth{\kl(P_{X_{n+1}|X^{n}} \| Q_{X_{n+1}|X^{n}})}
    \label{eq:risk-Q}
    \end{equation}
 where the KL divergence is $\kl(P\|Q) = \Expect_P[\log \frac{dP}{dQ}]$ if $P\ll Q$ and $\infty$ otherwise. 
 The minimax prediction risk  is then defined as
    \begin{equation}
    \Risk(\calP_{n+1}) \triangleq \inf_{Q_{X_{n+1}|X^{n}}}  \Risk(Q_{X_{n+1}|X^{n}}; \calP_{n+1}),
    \label{eq:risk}
    \end{equation}  
As exemplary applications, we consider two model classes with infinite memory: HMMs and renewal processes.
Relevant notations are deferred to \Cref{appdx: lemmas}.

\paragraph{Hidden Markov Models}

A hidden Markov process is obtained by passing a Markov process through a memoryless noisy channel. It provides a useful tool for modeling practical data such as natural language and speech signals. Specifically, fix $k,\ell \in \naturals$. Let $\{Z_t: t \geq 1\}$ be a stationary Markov chain on the state space $[k]\triangleq\{1,\dots,k\}$ with transition matrix $M$, which is a $k\times k$ row-stochastic matrix. 
Let $T$ denote a probability transition kernel from $[k]$ to $[\ell]$, that is, a $k\times \ell$ row-stochastic matrix.
Let $\{X_t: t \geq 1\}$ be an $[\ell]$-valued process such that for any $n$, $P_{X^n|Z^n} = \prod_{t=1}^n T(x_t|z_t)$. 
We refer to 
$\{X_t\}$  as a hidden Markov process  with \textit{transition probabilities} $M$ and \textit{emission probabilities} $T$, while $\{Z_t\}$ are called the \textit{hidden (or latent) states}.

Let $\PHMM_{n}(k,\ell)$ denote the collection of joint distributions of hidden Markov processes of length $n+1$ with state space $[k]$ and observation space $[\ell]$. This is a finite-dimensional parametric model (by $M$ and $T$) with a total of $k(k-1)+k(\ell-1)$ parameters.
We note that over this class the parameters $M$ and $T$ are \textit{not} identifiable since no further conditions such as full rank of $M$ are assumed cf.~\cite[Example 1]{alexandrovich2016nonparametric}. Yet, the prediction problem is both well-defined and non-trivial. We define the optimal prediction risk of HMM as follows:
\[
\RiskHMM(n,k,\ell) \triangleq
\Risk(\PHMM_{n}(k,\ell))
\]

Our main results on predicting hidden Markov processes are as follows.
\begin{theorem}[Optimal prediction risk for HMM]\label{thm: stats bounds}
The following holds:
\begin{itemize}
    \item There exists a universal constant $C$ such that for all
    $n\geq C k(k+\ell)$,
        \begin{equation}\label{eq: hmm opt risk}
  \risk_{\operatorname{HMM}}(n, k, \l)\leq C\left(\frac{k\l}{n}\log\frac{n}{k\l}+\frac{k^2}{n}\log\frac{n}{k^2}\right),
    \end{equation}
    achieved by an $n^{O(k^2+k\l)}$-time algorithm.

    \item Conversely, if either $\l\geq k$ and $n\geq k\l$, or $n \ge k^C$ and $k, \l\geq 2$, then
        \begin{equation}\label{eq: hmm opt risk2}
         \risk_{\operatorname{HMM}}(n, k, \l)\geq 
         C^{-1}\left(\frac{k\l}{n}\log\frac{n}{k\l}+\frac{k^2}{n}\log\frac{n}{k^2}\right).
    \end{equation} 
\end{itemize}
\end{theorem}
We note that before this work even for the simplest case of binary-state binary-emission HMMs, the best known result is $O(\frac{1}{\log n})$ by \cite{sharan2018prediction}, who considered prediction in HMMs under a somewhat different formulation than \prettyref{eq:risk} (with further averaging over the sample size $n$ and in the weaker total variation loss than KL; see \Cref{appdx: losses} for a detailed comparison). In comparison,  \prettyref{thm: stats bounds} shows that for $k,\ell = O(1)$, the optimal rate in KL is $\Theta(\frac{\log n}{ n})$ and attainable in polynomial time. Furthermore, we point out that our results, both for the lower bound and the upper bound, can be extended to HMMs with discrete state space but arbitrary observation space (Corollary~\ref{cor: any emission} and \ref{cor:general_emission_lb}). For instance, as a side result, we determine optimal prediction risks for HMM with Gaussian emissions in terms of the output dimension (Remark~\ref{rmk: gaussian risk}).

Since HMM has infinite memory, a natural idea is to first approximate it by a finite-order Markov chain then invoke existing prediction risk bounds for Markov models; this was the key insight in \cite{sharan2018prediction}. However, this approach based on Markov approximation does not achieve the optimal risk bound. Indeed, it was shown in \cite{HJW23} that the optimal prediction risk for order-$d$ Markov chains on $[\ell]$ scales as $\Theta(\frac{\ell^d}{n} \log \frac{n}{\ell^d} )$, already much larger than the risk in \Cref{thm: stats bounds} for moderate $d$. 
Instead, our approach in Section \ref{sec:riskred} applies ideas from universal compression, in particular, the redundancy to control the complexity of HMMs, while introducing an additional memory term to handle the long-range dependence of the HMM. 
The overall algorithm is based on dynamic programming that averages state sequences of length $n$. 

On the other hand, for large $k$ or $\ell$, the statistically optimal algorithm in \Cref{thm: stats bounds} based on dynamic programming is no longer efficient. Next we give a polynomial-time algorithm that achieves a prediction risk vanishing at a suboptimal rate. In contrast to \Cref{thm: stats bounds}, this efficient algorithm is built upon an order-$O(\log n)$ Markov approximation.



\begin{theorem}[Computationally efficient algorithms]\label{thm:computation_UB} 
There exists a polynomial-time estimator whose $\kl$ prediction risk over $\PHMM_n(k,\ell)$ is $O(\frac{\log k\log\l}{\log n})$, provided that  
$\log k \log\l=o(\log n)$. 
\end{theorem}
A similar rate was also established in average TV loss using Markov approximation in \cite{sharan2018prediction}. However, their result uses empirical averages to estimate Markov transitions and applies martingale concentration results, making it hard to generalize to e.g. $\kl$. Here, our result applies a much simpler approach via redundancy of the ``add-one'' code whose $\kl$ risk can be controlled.

We also complement the above upper bound with computational lower bounds in HMMs, showing that the prediction risk for any $\operatorname{poly}(n)$-time algorithm is $\Omega(1/\log\log n)$ if $\log(k\ell) = \Omega(\log n)$. 
\begin{theorem}[Informal: Computational lower bounds]\label{thm:computation_LB}
 The following holds under certain 
 cryptographic  hardness assumptions:
     \begin{enumerate}
         \item For any $\eps>0$, $k\geq \log^{1+\eps} n$, no $\operatorname{poly}(n)$ algorithm achieves $o(\frac{\log k}{\log n\log\log n})$ risk for $\ell\geq 2$. 
         \item For every $\alpha>0$ there exists $k_\alpha\geq 2$, such that if $k\geq k_\alpha$ and $\l\geq n^\alpha$, no $\operatorname{poly}(n)$ algorithm can achieve $o(1)$ risk.
     \end{enumerate}
\end{theorem}
Our lower bounds are proven by showing that certain cryptographic structures can be embedded into an HMM with a limited number of states or emission space. Such embedding was studied extensively in prior works (e.g. \cite{mossel2005learning,sharan2018prediction}).
\paragraph{Renewal processes}

As another application of our techniques, we turn to the class of renewal processes. A natural example of predicting a renewal process may be described as follows:
Suppose that for a given driver the time (in days) between consecutive traffic accidents are random and i.i.d. Given the driving records (safety or accident) for the past $n$ days, the insurance company seeks to predict the probability of an accident occurring on the next day, where the interarrival distribution is unknown.


To give a formal definition of a renewal process, let $T_0,T_1,T_2,\ldots$ denote a sequence of independent $\naturals$-valued random variables, where $T_i$ are iid drawn from some distribution $\mu$ with a finite mean. 
A \textit{renewal process} $\{X_t: t\geq 1\}$ is binary valued such that
$\{t:X_t=1\}$ is exactly $\{T_0, T_0+T_1, T_0+T_1+T_2,\dots\}$.
We refer to $T_0$ and $\{T_i: t\geq 1\}$ as the initial wait time and the interarrival times.
It is known (\cite{csiszar96redundancy}) that $\{X_t\}$ is stationary if and only if $T_0$ is distributed as
$\P(T_0=t)=\frac{1}{\E_\mu[T_1]}\sum_{s\geq t}\mu(s), t\in\naturals$.

Let $\Prenew_{n}$ denote the collection of joint distributions of a stationary renewal process of length $n+1$ with a finite expected interarrival time.
In contrast to the previously considered HMM, this is a nonparametric (infinite-dimensional) model parameterized by the interarrival time distribution $\mu$.
Particularizing \prettyref{eq:risk}, define the optimal prediction risk as
$\Riskrenew(n) \triangleq \Risk(\Prenew_{n})$.
The following result determines its sharp rate:
\begin{theorem}[Prediction of renewal processes]\label{thm: risk renewal}
    There exists an absolute constant $C$, such that
$$C^{-1}\sqrt{n^{-1}}\leq
\Riskrenew(n)
\leq  C\sqrt{n^{-1}}.$$
\end{theorem}
The proof of this result builds upon the redundancy bound $\Theta(\sqrt{n})$ (\cite{csiszar96redundancy,flajolet02analytic}) for renewal processes.
Both a strength and a weakness of our theory, the predictor attaining the optimal rate of $1/\sqrt{n}$ is not computationally efficient (see 
\prettyref{sec:ub-renewal}) and it is unclear how to do so in polynomial time. One idea may be the following. For an oracle who knows the true interarrival distribution $\mu$, it can determine the true $P_{X_{n+1}|X^n}$ by the \textit{hazard rate}: 
\begin{equation}
P_{X_{n+1}=1|X^n} = {\mu(\tau+1)}/{\sum_{t> \tau} \mu(t)},
\label{eq:hazard}
\end{equation}
 where $\tau$ is the time till the most recent renewal (or the origin of time). Thus a natural idea is to replace $\mu$ by its empirical version if sufficiently many renewals are observed, and predict $P_{X_{n+1}=1|X^n}$ by some small probability, e.g. $\frac{1}{\operatorname{poly}(n)}$, 
otherwise. The analysis of this algorithm, however, appears challenging absent  assumptions on the distribution $\mu$.

\subsection{Related works}
The connections between compression and prediction are long studied dating back to e.g.~\cite{Rissanen84universal,feder1992universal,haussler1998sequential}. More recently, a line of works (\cite{Falahatgar16learning,HJW23}) determined the optimal prediction risk in KL for Markov models up to constants and showed that it is near the ``parametric rate'' $\frac{K}{n}$, where $K$ is the number of model parameters, but is strictly slower by logarithmic factors. A key assumption of these results is the \emph{finite memory} of the true model, where the next observation may depend on only the most recent few.

Turning to HMMs, the majority of works in the statistical learning literature focus on identifiability (\cite{alexandrovich2016nonparametric,huang2015minimal}) and 
parameter estimation, using algorithms include moments or tensor methods (\cite{de2017consistent,anandkumar2014tensor,sharan2017overcomplete,abraham2022fundamental}) and penalized likelihood (\cite{de16minimax} \cite{lehericy2021nonasymptotic}). However, the success of those methods routinely requires extra assumptions on parameters such as spectral properties (\cite{huang2015minimal,abraham2022fundamental}) and sparsity (\cite{sharan2017overcomplete}). For prediction, we need not and do not impose these assumptions. 
In terms of prediction, the closest work we are aware of is \cite{sharan2018prediction}, where the authors focus on algorithms via Markov approximation. Finally, computational barriers (of various forms) are known to exist for both prediction and estimation of HMMs (\cite{mossel2005learning,sharan2018prediction}).


\section{Prediction risk bound based on universal compression}
\label{sec:riskred}
Having defined the prediction risk \prettyref{eq:risk} for a general model class $\calP_{n+1}$, we introduce 
 the closely related redundancy problem which is at the heart of both theory theory and algorithms for universal compression. The \emph{redundancy} of a joint distribution $Q_{X^{n+1}}$ (often referred to as a 
\textit{probability assignment})
is defined as the worst-case KL risk of fitting the joint distribution of $X^n$, namely
\begin{equation}
    \Red(Q_{X^{n+1}}; \calP_{n+1}) \triangleq \sup_{P_{X^{n+1}}\in\calP_{n+1}} \kl(P_{X^{n+1}} \| Q_{X^{n+1}}).
    \label{eq:red-Q}
    \end{equation}       
    Optimizing over the probability assignment $Q_{X^{n+1}}$, the minimax redundancy is defined as
    \begin{equation}
    \Red(\calP_{n+1}) \triangleq \inf_{Q_{X^{n+1}}} \Red(Q_{X^{n+1}}; \calP_{n+1}),
    \label{eq:red}
    \end{equation}      
The role of a probability assignment in universal compression is a simultaneous approximation to a class of models.
It is known that the Shannon entropy $H(P_{X^{n+1}})$ is within 1 bit of the best average code length for the optimal compressor that knows the source distribution $P_{X^{n+1}}$. The goal of universal compression is to design a compressor that simultaneously approaches the entropy for a class of models. This can be achieved by applying the compressor (e.g.~arithmetic coding) designed for a probability assignment $Q_{X^{n+1}}$, whose excess code length over $H(P_{X^{n+1}})$ is at most within 1 bit of $\Red(Q_{X^{n+1}})$ for all $P_{X^{n+1}}$ in the class $P_{n+1}$. Thanks to this reduction, the design of universal compressor is largely reduced to choosing a good probability assignment
and the redundancy is the central quantity in universal compression.

The following result
    relates the redundancy and the prediction risk for any stationary data-generating process.
     In the case of i.i.d.~models, 
 this type of reduction relating cumulative risks and individual risks  is known as \textit{online-to-batch conversion} which, in the present context, dates back at least to \cite{yang1999information} for density estimation
 (see e.g.~\cite[Proposition 32.7]{PW-it} for a summary).
\begin{proposition}[Upper bound prediction risk by redundancy]
\label{prop:riskred}     
Suppose that each $P_{X^{n+1}}\in\calP_{n+1}$ is stationary, that is, 
$P_{X_{t_1},\ldots,X_{t_k}} = 
P_{X_{t_1+t},\ldots,X_{t_k+t}}$ 
for any shift $t\geq 1$ and $1 \leq t_1,\ldots,t_k\leq n+1-t$.
Let $Q_{X^{n+1}}$ be an arbitrary joint distribution factorizing as $Q_{X^{n+1}}=\prod_{t=1}^{n+1} Q_{X_t|X^{t-1}}$.
Consider an estimator 
$\widetilde Q_{X_{n+1}|X^n}$ defined as
            \begin{equation}
        \widetilde Q_{X_{n+1}|X^n}(\cdot|x^n) \triangleq \frac{1}{n} \sum_{t=1}^{n} Q_{X_{t+1}|X^t}(\cdot|x_{n-t+1}^n)
        \label{eq:Q1}
        \end{equation}
        Then
        \begin{equation}
        \Risk(\widetilde Q_{X_{n+1}|X^n}; \calP_{n+1}) \leq \frac{1}{n} \Red(Q_{X^{n+1}}; \calP_{n+1}) +  \frac{1}{n} \sum_{t=1}^n I(X_{n+1};X^{n-t}|X^n_{n-t+1}).
        \label{eq:riskred-ach1}
        \end{equation}
\end{proposition}
Since the last term in \prettyref{eq:riskred-ach1} does not depend on the probability assignment $Q$, taking the supremum over the worst-case $P$ in class then optimizing over $Q$ yields
\begin{equation}
\Risk(\calP_{n+1}) \leq \frac{1}{n} \Red(\calP_{n+1}) + \mem(\calP_{n+1})
        \label{eq:riskred-minimax}
        \end{equation}
  where
  the residual term
  \begin{equation}      
  \mem(\calP_{n+1}) \triangleq \sup_{P_{X^{n+1}} \in \calP_{n+1}}
  \frac{1}{n} \sum_{t=1}^n I(X_{n+1};X^{n-t}|X^n_{n-t+1})
  \label{eq:mem}
  \end{equation}
  measures the memory, in a average sense, of the data-generating process in the model class. 
  Indeed, recall that the conditional mutual information $I(A;B|C)$ measures the conditional dependency between $A$ and $B$ given $C$, and is zero if they are conditional independent.
  Thus, for Markov models of order $m$,
  $I(X_{n+1};X^{n-t}|X^n_{n-t+1})=0$ for all $t\geq m$ and $\mem(\calP_{n+1}) $ is at most $O(\frac{mH(X_{n+1})}{n})$. As a result, for bounded $m$ we get 
  $\Risk \lesssim \frac{\Red}{n}$.\footnote{In fact, a slightly different argument in 
  \cite[Lemma 6]{han2021optimal} avoids the additive error term and shows 
  $\Risk(\calP_{n+1}) \leq \frac{1}{n+1-m} \Red(\calP_{n+1})$ for $m$th-order Markov models.}
  For models with infinite memory, such as HMMs and renewal processes, applying this redundancy-based risk bound requires bounding the memory term uniformly, which we carry out in the subsequent sections. 

We end this section with a couple of remarks. First, applying the risk bound in \prettyref{prop:riskred} relies on bounding the redundancy of a model class from above, which is often achieved by further relaxing the redundancy, an approach known as \textit{individual sequences}.
Replacing the expectation in $\kl(P_{X^{n+1}} \| Q_{X^{n+1}}) = 
\Expect_P\left[\log \frac{P_{X^{n+1}} }{ Q_{X^{n+1}}}\right]$ by the maximum, one arrives at the so-called \textit{minimax pointwise redundancy}
\begin{equation}
\Red(\calP_{n+1}) \leq \widetilde\Red(\calP_{n+1}) \triangleq \inf_{Q_{X^{n+1}}} \sup_{P_{X^{n+1}}\in\calP_{n+1}} 
\max_{x^{n+1} \in \calX^{n+1}}
\log \frac{P_{X^{n+1}}(x^{n+1}) }{ Q_{X^{n+1}}(x^{n+1})}.
    \label{eq:red-pointwise}
    \end{equation}
The optimal probability assignment for \prettyref{eq:red-pointwise}
is known as Shtarkov's \textit{normalized maximum likelihood} assignment
$Q^*_{X^{n+1}}(x^{n+1}) \propto \sup_{P_{X^{n+1}}\in\calP_{n+1}} P_{X^{n+1}}(x^{n+1})$, leading to the following formula for the minimax pointwise redundancy
as a \textit{Shtarkov sum}
\begin{equation}
\widetilde\Red(\calP_{n+1})= \log 
\sum_{x^{n+1} \in \calX^{n+1}} \sup_{P_{X^{n+1}}\in\calP_{n+1}} P_{X^{n+1}}(x^{n+1}).
\label{eq:shtarkov}
\end{equation}
Most redundancy bounds, including those that we apply (\cite{csiszar96redundancy}), are obtained by either analyzing the pointwise redundancy or 
directly bounding the above Shtarkov sum.
This combinatorial approach 
avoids all probabilistic computation 
and is essentially the reason why one can sidestep mixing conditions in HMMs.

    Second,    for i.i.d.~models, say, distributions over $k$ elements, the upper bound on the prediction risk in 
    \prettyref{prop:riskred} is in fact loose by a logarithmic factor, since we know that 
     $\Risk \asymp \frac{k-1}{n}$
     and $\Red \asymp (k-1)\log n$.
     Interestingly, the compression-prediction method seems particularly effective for models with memory, which is tight up to \textit{constant factors} for finite-order Markov chains (\cite{HJW23})
     and, as we show in this paper, HMMs and renewal processes. 
     Complementing \prettyref{prop:riskred}, we give a reduction argument that shows the prediction risk of a given class of HMMs is lower bounded by the redundancy of a slightly smaller subclass -- see \prettyref{sec:reduction} for details.

\section{Proof of the upper bounds}\label{sec: upper bounds}
In this section, we make use of Proposition \ref{prop:riskred} to upper bound the prediction risk for HMMs and renewal processes. This entails upper bounding the minimax redundancy $\red(\calP)$ in \eqref{eq:red} and the memory term $\mem(\calP)$ in \eqref{eq:mem}, for both HMMs and renewal processes. 

\subsection{Bounding the memory term for HMMs}
\label{sec:ub-hmm}
We start with a simple upper bound on the memory term in 
\eqref{eq:riskred-minimax} for HMMs. Similar bounds have appeared previously in the literature, see, e.g., \cite[p.~932]{birch1962approximations}. 
\begin{proposition}
\label{prop: decay info}
Let $\{X_t\}$ be a stationary hidden Markov process.  Then
   $$\sum_{t=1}^n I(X^{n-t}; X_{n+1}|X_{n-t+1}^n) \leq I(Z_1; X^{n+1}).$$
\end{proposition}

Suppose there are at most $k$ latent states. Then 
 $I(Z_1; X^{n+1}) \le H(Z_1)\le \log k$ regardless of the emissions. 
Applying \prettyref{prop: decay info} to \eqref{eq:riskred-ach1} yields
\begin{equation}\label{eq: riskred strong}\risk(\widetilde Q_{X_{n+1}|X^n}) \leq \frac{1}{n}\redd(Q_{X^{n+1}})+
\frac{\log k}{n},
\end{equation}
where $Q_{X^{n+1}}$ is any probability assignment and 
$\widetilde Q_{X_{n+1}|X^n}$ is the predictor defined in \prettyref{eq:Q1}.
As we show next, the memory term turns out to be negligible compared with the redundancy. 


\subsection{Redundancy bound for HMM}
Next we upper bound the redundancy $\red(\PHMM_n(k,\ell))$ and prove the upper bound in \Cref{thm: stats bounds}. To this end, it suffices to bound the redundancy of the joint state-emission sequence. Indeed, by definition \prettyref{eq:red}, for any $Q_{X^{n+1},Z^{n+1}}$ and any $P_{X^{n+1},Z^{n+1}}$ in the model class, we have 
$\kl(P_{X^{n+1},Z^{n+1}}\|Q_{X^{n+1},Z^{n+1}}) 
\leq \kl(P_{X^{n+1}}\|Q_{X^{n+1}})$. 
Let us define a joint probability assignment
by separately approximating the 
transition and emission probabilities using the probability assignment designed for the Markov and i.i.d.~class respectively:
\begin{align}\label{eq:Q_joint_distribution}
Q_{X^{n+1}, Z^{n+1}}(x^{n+1}, z^{n+1}) &= Q_{Z^{n+1}}(z^{n+1})\cdot Q_{X^{n+1}|Z^{n+1}}(x^{n+1}|z^{n+1}) \\
&= \frac{1}{k}\prod_{t=1}^n M_t(z_{t+1}|z_t)\cdot \prod_{t=1}^{n+1} T_{t}(x_{t}|z_{t}), \nonumber
\end{align}
where $M_t$ and $T_t$ are the add-one estimators (\cite{Krichevsky81performance})
for the transition and emission probabilities, respectively:
\begin{align}
M_t( z' | z) &= \frac{1+\sum_{i=1}^{t-1} \one_{z_{i+1}=z'\text{ and }z_{i}=z}}{k+\sum_{i=1}^{t-1}\one_{z_i=z}}, \label{eq: add-one M}\\
T_{t}(x|z)&=\frac{1+\sum_{i=1}^{t-1}\one_{z_i=z\text{ and }x_i=x}}{\l+\sum_{i=1}^{t-1}\one_{z_i=z}}. \label{eq: add-one T}
\end{align}
Finally, let $Q_{X^{n+1}}$ be the marginal of \prettyref{eq:Q_joint_distribution}.
The following result bounds on the pointwise redundancy of $Q_{X^{n+1}, Z^{n+1}}$ and thus that of $Q_{X^{n+1}}$. By \prettyref{eq:red-pointwise}, this also bounds their average-case redundancy.


\begin{proposition}\label{prop: red upper bound}
Let $n\ge k(k+\ell)$. For any hidden transition matrix $M$ (with stationary distribution $\pi$) on state space $[k]$ and emission matrix $T$ from $[k]$ to $[\ell]$, 
$$
\max_{x^{n+1}, z^{n+1}}\log \frac{\pi\left(z_1\right) \prod_{t=1}^{n} M\left(z_{t+1} | z_t\right) \prod_{t=1}^{n+1} T\left(x_t | z_t\right)}{ Q_{X^{n+1}, Z^{n+1}}(x^{n+1}, z^{n+1}) } \lesssim k^2 \log \frac{n}{k^2}+k \ell \log \frac{n}{k \ell}.
$$
Consequently, $\Red(Q_{X^{n+1}}; \PHMM_n(k,\ell))\lesssim k^2\log(n/k^2) + k\ell\log(n/k\ell)$. 
\end{proposition}
Combining Propositions \ref{prop:riskred}, \ref{prop: decay info}, and \ref{prop: red upper bound}, we have
\begin{align*}
\Risk(\PHMM_n(k,\ell)) \lesssim \frac{k^2}{n} \log \frac{n}{k^2}+\frac{k \ell}{n} \log \frac{n}{k \ell} + \frac{\log k}{n}, 
\end{align*}
which completes the upper bound proof of \Cref{thm: stats bounds}. In fact, the same program can be extended to HMMs with general emissions. To this end, let $X$ take value in a general space $\calX$, and $\calQ$ be a class of probability distributions over $\calX$. We use $\PHMM_n(k,\calQ)$ to denote the collection of stationary HMMs of length $n+1$, with hidden states in $[k]$ and emissions in $\calQ$ (i.e. $P_{X|Z}(\cdot|z)\in \calQ$ for all $z\in [k]$). The following corollary, proved in Appendix \ref{sec:any emission}, bounds the prediction risk:
\begin{corollary}\label{cor: any emission}
Suppose $\Red(\calQ^{\otimes t})\le R(t)$ for all $t$ for some concave $R(\cdot)$. 
Then for $n\gtrsim k^2$,
\begin{align*}
\Risk(\PHMM_n(k,\calQ)) \lesssim \frac{k^2}{n}\log\frac{n}{k^2} + \frac{k}{n}R\left(\frac{n+1}{k}\right). 
\end{align*}
\end{corollary}
\begin{remark}\label{rmk: gaussian risk}
    When $\calQ$ is the set of Gaussian distributions $\N(w, I_d)$ with $w\in [-1, 1]^d$, one has $\Red(\calQ^{\otimes n})\lesssim d\log n$ by Gaussian channel capacity. Hence, Corollary \ref{cor: any emission} shows that the optimal prediction risk for Gaussian HMM with $k$ hidden states is $O(\frac{k(k+d)}{n}\log n)$. Furthermore, as we will see in the next section (Corollary~\ref{cor:general_emission_lb}), this bound is tight.
\end{remark}

\subsection{An optimal prediction algorithm}
We show that the estimator in \Cref{thm: stats bounds} can be computed in time $n^{O(k^2+k\l)}$, and it suffices to prove that the marginal distribution $Q_{X^{n+1}}$ can be efficiently computed based on the joint distribution $Q_{X^{n+1},Z^{n+1}}$ in \eqref{eq:Q_joint_distribution}. Our idea relies on an equivalent expression of $Q_{X^{n+1},Z^{n+1}}$ via sufficient statistics: let $M\in \mathbb{R}^{k\times k}, T\in \mathbb{R}^{k\times \ell}$ be the transition and emission count matrices, formally defined as $M_{z,z'}=\sum_{t=1}^n \one_{z_t=z,z_{t+1}=z'}$ and $T_{z,x} = \sum_{t=1}^{n+1} \one_{z_t = z, x_t = x}$, then
\begin{align*}
Q_{X^{n+1},Z^{n+1}}(x^{n+1},z^{n+1}) = \frac{1}{k}\prod_{z\in [k]}\left( \frac{\prod_{z'\in [k]} M_{z,z'}!}{k^{\overline{M_z}}} \cdot \frac{\prod_{x\in [\ell]} T_{z,x}!}{\ell^{\overline{T_z}}} \right)\triangleq F(M, T), 
\end{align*}
where $M_z, T_z$ denotes the row sums of $M, T$, and $k^{\overline{m}} = k(k+1)\cdots(k+m-1)$ is the rising factorial. Based on the above expression, to compute the marginal distribution $Q_{X^{n+1}}$ it suffices to enumerate over all possible matrices $(M, T)$ and compute the number $\calA(M,T; x^{n+1})$ of sequences $z^{n+1}$ that induce a given $(M,T)$.  The following lemma, proved in \prettyref{appdx: small k l}, shows that for each $(M,T)$ 
this enumeration can be done in $n^{O(k^2+k\ell)}$ time by dynamic programming. 

\begin{lemma}\label{lem: const k l alg main}
Given any sequence $x^{n+1}$, the count $\calA(M,T; x^{n+1})$ can be computed in time $n^{O(k^2+k\ell)}$. 
\end{lemma}

Since the entries of $(M,T)$ take values in $\{0,1,\cdots,n+1\}$, the number of all possible matrices is $n^{O(k^2+k\ell)}$.
This completes the proof of the computational upper bound in Theorem \ref{thm: stats bounds}. 

\subsection{Renewal processes}
\label{sec:ub-renewal}
For the class $\Prenew_n$ of renewal processes defined in \prettyref{sec:main}, we invoke a well-known result on its redundancy: 
\begin{lemma}[\cite{csiszar96redundancy}]\label{lem: renewal red}
 $\Red(\Prenew_{n})=\Theta(\sqrt{n})$.
 \end{lemma}

By Proposition \ref{prop:riskred} and Lemma \ref{lem: renewal red}, it remains to upper bound the memory term $\mem(\Prenew_n)$. 
A stationary renewal process $\{X_t\}$ with interarrival distribution $\mu$ is represented by a stationary HMM with a countably infinite state space as follows: 
\begin{enumerate}
    \item The hidden states $\{Z_t\}$ takes values in $\naturals$ represents the ``countdown'' until the next renewal, where $\P(Z_{t+1}=i-1|Z_t=i)=1$ if $i\geq 2$ and $\P(Z_{t+1}=j|Z_t=1)=\mu(j)$ for  $j\geq 1$.
    
    \item The emissions is binary and deterministic: $X_t=\one_{Z_t=1}$. 
\end{enumerate}
Furthermore, the stationary state distribution $\pi_\mu$
is precisely the law of the initial wait time in \prettyref{sec:main}, given by $\pi_\mu(t) = \frac{\sum_{i\geq t} \mu(i)}{\sum_{i\geq 1} i \mu(i)}$. 
This HMM representation allows us to apply  Proposition \ref{prop: decay info} to bound the memory term
\begin{align*}
 \mem(\Prenew_n) \le \frac{1}{n}I(Z_1; X^{n+1}). 
\end{align*}
Although $Z_1$ takes infinitely many values, we show that the above mutual information is still at most $O(\log n)$: Let $\widetilde Z_1\triangleq\min\{Z_1, n+2\}$. Then $Z_1\to \widetilde Z_1\to X^{n+1}$ is a Markov chain because $\widetilde{Z}_1 = Z_1$ if $\widetilde{Z}_1<n+2$, and $X^{n+1}=0^{n+1}$ if $\widetilde{Z}_1= n+2$.
Therefore, by the data processing inequality: 
$$I(Z_1; X^{n+1})\leq I(\widetilde Z_1; X^{n+1})\leq H(\widetilde Z_1)\leq \log (n+2).$$ 
The upper bound of \Cref{thm: risk renewal} then follows from Proposition~\ref{prop:riskred} and Lemma~\ref{lem: renewal red}. 

Note that the redundancy upper bound in \prettyref{lem: renewal red} is obtained by analyzing the pointwise redundancy \prettyref{eq:red-pointwise} and bounding the Shtarkov sum \prettyref{eq:shtarkov} by  the partition number whose asymptotics yields the $\sqrt{n}$ term 
(\cite{hardy1918asymptotic}).
Thanks to \prettyref{prop:riskred}, 
averaging of the conditionals of the Shtarkov distribution (normalized maximum likelihood) yields a predictor that attains the optimal rate $\frac{1}{\sqrt{n}}$.
As discussed in \prettyref{sec:main}, finding a computationally efficient optimal predictor is an interesting open question. 


\section{Proof of the lower bounds}
\label{sec:lb}
This section proves the lower bounds of the prediction risk for HMMs with further technical results deferred till \Cref{appdx: lower HMM}. We first present a generic embedding idea to lower bound the prediction risk using the redundancy of a slightly smaller class of HMMs; this reduction essentially shows the tightness of the compression-prediction program in Section \ref{sec:riskred} when a (hidden) Markov structure is available. Next we lower bound the redundancy $\Red(\PHMM_n(k,\ell))$. 
For renewal processes we use an explicit prior and lower bound the Bayes prediction risk directly (see \Cref{appdx: renewal}). 

\subsection{Reduction from redundancy to prediction risk}\label{sec:reduction}

Complementing the upper bound in Proposition \ref{prop:riskred}, 
the following result lower bounds the prediction risk by the redundancy of HMMs with one fewer states and observations.
\begin{proposition}[Lower bound prediction risk by redundancy]\label{prop:risk_lower_bound}
    For $\PHMM_n(k,\ell)$ with $k\ge 2, \l\geq 3$, 
    $$\risk_{\operatorname{HMM}}(n, k, \l)\gtrsim\frac{1}{n}( \red(\PHMM_{n+1}(k-1,\ell-1))-\log\l).$$
\end{proposition}
For $k\ge 2,\ell \ge 3$, combining this result with 
the redundancy lower bound in Theorem \ref{thm: redd lower discr} proves the lower bound in Theorem \ref{thm: stats bounds}. (Note that for $k=2$, $\PHMM_{n+1}(k-1,\ell-1)$ is in fact an i.i.d.~process over $[\ell]$ and has redundancy $\Theta(\ell\log(n/\ell))$ for $n\gtrsim \ell$ \cite{Davisson73universal}). 

The proof of Proposition \ref{prop:risk_lower_bound} relies on a reduction from redundancy to prediction risk. Given an arbitrary instance $Q$ of the HMM parameters with hidden alphabet $[k-1]$ and emission $[\l-1]$, we seek to construct another instance $P$ for the HMM parameters with hidden alphabet $[k]$ and emission $[\ell]$. The main idea is to add a ``lazy'' state $k$ that slows down the chain. This uninformative state has a heavy self loop such that with constant probability, the chain only explores the original state space $[k-1]$ for a period of time that is approximately uniform in $[n]$, effectively reducing the sample size from $n$ to $\operatorname{Unif}([n])$. As such, the prediction risk can then be related to the cumulative risk, that is, the redundancy. Specifically, define
\begin{enumerate}
    \item Emission probabilities: $P(X=\ell | Z = k) = 1$ and $P_{X|Z=z} = Q_{X|Z=z}$ for all $z\in [k-1]$. In other words, state $k$ always emits $\ell$, while the emissions of other states are the same as $Q$. 
    \item Transition: let $\pi_Q$ be the stationary distribution over the state space $[k-1]$ under $Q$: 
    \begin{align}
    P(Z_2 = j | Z_1 = i) = \begin{cases}
       \one_{j=k}(1 - 1/n)+\one_{j\neq k}\pi_Q(j)/n &\text{if } i = k,\\
       1/n & \text{if } i \neq k, j = k, \\
       (n-1)Q(Z_2 = j | Z_1 = i) / n & \text{if } i\neq k, j\neq k.
    \end{cases}
    \label{eq:reduction_transition}
    \end{align}
\end{enumerate}

One can verify that the stationary state distribution of the HMM $P$ is 
$\pi_P(k)=1/2$ and $\pi_P(i)=\pi_Q(i)/2$. For $0\le t\le n-1$, define the event $E_t=\{x^n: x^t=\l^t, x_{t+1}^n\in [\l-1]^{n-t}\}$. A simple computation shows that $\P(E_t)=\Theta\left(\frac{1}{n}\right)$
for all $1\le t\le n-1$, and $\P(E_0)= \Theta(1)$.

Next we consider a general prior distribution of $Q$, which induces a prior of $P$. Note that conditioned on the event $E_t$, the Bayes prediction risk of $P_{X_{n+1}|X^n}$ (or equivalently $I(P; X_{n+1}|X^n, E_t)$) equals to the Bayes prediction risk of $Q_{X_{n+1}|X_{t+1}^n}$ (or equivalently $I(Q;X_{n+1}|X_{t+1}^n)$) times $1-1/n$, the scaling factor between the transition probabilities under $Q$ and $P$ on $[k-1]$.

Therefore, the overall Bayes prediction risk of $P_{X_{n+1}|X^n}$ is lower bounded by
     \begin{align*}
        \sum_{t=0}^{n-1} \P(E_t)I(P; X_{n+1}| X^n, E_t) &= \sum_{t=0}^{n-1} \P(E_t)\cdot \left(1-\frac{1}{n}\right)I(Q; X_{n+1}| X_{t+1}^n) \\&\gtrsim \frac{1}{n}\sum_{t=0}^{n-1} I(Q;X_{n-t+1}|X^{n-t}) \\ &= \frac{1}{n}\left(I(Q; X^{n+1})-I(Q; X_{1})\right)\geq \frac{1}{n}\left(I(Q; X^{n+1})-\log \l\right).
     \end{align*}
Maximizing over the prior distributions of $Q$ leads to redundancy and proves Proposition \ref{prop:risk_lower_bound}. 

The simple embedding above is a bit wasteful as it designates a special emission symbol to signify the lazy state. As such, 
the case of $\ell=2$ is out of reach. 
Applying more delicate reductions, the next result (proved in Appendix \ref{sec:riskred-improved}) gives a risk lower bound based on the redundancy of HMMs with the same observation space, 
with the additional constraint that the stationary state distribution is uniform, which, for large $k$, has the same redundancy within constant factors. Thus this result is applicable to the case of binary and even continuous emissions such as Gaussians (Remark~\ref{rmk: gaussian risk}). 

\begin{proposition}\label{prop:risk_lb_general}
In the context of Corollary \ref{cor: any emission}, for all $k\ge 2$ it holds that
\begin{align*}
\risk(\PHMM_n(k,\calQ)) \gtrsim \frac{1}{n}\red(\PHMM_{n+1,\mathsf{U}}(k-1,\calQ)) - \frac{\log(nk) + \red(\calQ)}{n}, 
\end{align*}
where $\PHMM_{n,\mathsf{U}}(k,\calQ)$ is the set of all stationary HMMs with hidden states in $[k]$, emission distributions in $\calQ$, and a uniform stationary distribution for the hidden states.  
\end{proposition}
\begin{corollary}\label{cor:general_emission_lb}
Suppose that there are constants $0\le c_1<c_2\le 1$ and a map $f: \calX\to \{0,1\}$ such that for all $c\in [c_1, c_2]$, there exists $Q_c\in\cal Q$ such that $f_{\#}Q_c=\mathsf{Bern}(c)$. Then
\begin{align*}
\Risk(\PHMM_n(k,\calQ)) \gtrsim \frac{k^2\log n}{n} + \frac{k}{n}\Red(\calQ^{\otimes \lfloor n/k \rfloor})
\end{align*}
as long as $k\ge C$ and $n\ge k^D$, where $(C,D)$ are absolute constants. 
\end{corollary}

\subsection{Lower bounding the redundancy of HMM}
The general idea of lower bounding the redundancy is via the variational representation $\Red(\calP) = \sup_{\mu} I(\theta; X)$, where $\theta\sim \mu$ is a random element of $\calP$ according to the distribution (prior) $\mu$, and conditioned on this element, the random variable $X$ follows the distribution $\theta$. Following \cite{Davisson81efficient}, lower bounding the 
mutual information $I(\theta;X)$ requires us to construct an estimator of $\theta$ that achieves small error on 
a sufficiently rich sub-model class (that we also need to construct), 
leading to high mutual information. 
For the challenging case of overcomplete HMM (e.g. $\ell=2$), this estimator is based on tensor decomposition.

\paragraph{Large $\ell$.} We start with the easy case of $\ell \ge k$, where the redundancy of the HMM mainly comes from the emission probabilities. The prior distribution $\mu$
is constructed as follows: the transition of the hidden states is a deterministic cycle $\operatorname{C}_k$ (i.e. $1\to 2\to \cdots\to k\to 1$), and the emission distributions are drawn independently: $(\theta_z)_{z\in[k]} \triangleq (P_{X|Z}(\cdot|z))_{z\in [k]}\sim \mu_0^{\otimes k}$, where $\mu_0$ is some prior distribution over $\calP([\ell])$, the collection of all probability measures on $[\ell]$. 
Therefore, 
\begin{align*}
I(\theta; X^n) = I(\theta; Z^n, X^n) - I(\theta; Z^n | X^n) &\stepa{\ge} I(\theta; X^n | Z^n) - H(Z^n) \\
&\stepb{\ge} \sum_{z\in [k]} I(\theta_z; (X_i: Z_i = z) | Z^n) - \log k \\
&\stepc{\ge} \sum_{z\in [k]} I(\theta_z; Y_z^{\lfloor n/k \rfloor}) - \log k = k I(\theta_1; Y_1^{\lfloor n/k \rfloor}) - \log k,
\end{align*}
where (a) is due to $I(\theta; Z^n) =0$, (b) follows from the mutual independence of $(\theta_z, (X_i: Z_i = z))_{z\in [k]}$ given $Z^n$, (c) introduces an auxiliary sequence of i.i.d.~random variables $Y_{z,1},Y_{z,2},\cdots\sim \theta_z$, and uses that the sample size of $Y_z$ is at least $\lfloor n/k \rfloor$ for each $z\in [k]$. Consequently, 
\begin{align*}
\Red(\PHMM_n(k,\ell)) = \sup_\mu I(\theta; X^n) \ge k\cdot \sup_{\mu_0} I(\theta_1; Y_1^{\lfloor n/k \rfloor}) - \log k \gtrsim k\ell\log \frac{n}{k\ell} - \log k
\end{align*}
for $n\gtrsim k\ell$, where we use the classical redundancy bound of i.i.d.~model $\Red(\calP([\ell])^{\otimes m}) = \Omega(\ell \log(m/\ell))$ if $m\gtrsim \ell$ (\cite{Davisson73universal}).  

\paragraph{Large $k$.} The analysis for the overcomplete case of large $k$ is far more challenging.
Without loss of generality consider $\ell = 2$. 
The lower bound on mutual information 
crucially relies on the following lemma, which shows that an estimator on tensor decomposition succeeds provided that the transition matrices are close to a deterministic cycle.

\begin{lemma}[Estimating the transition]\label{prop: redundancy lower bound}
    Let $0\leq t_1<t_2\leq 1$ be fixed constants. There exist positive constants $c_0,c_1,c_2,c_3$ and fixed $p_1,\cdots,p_k\in (t_1,t_2)$ such that when $k\ge c_0, n\ge k^{c_1}$ and:
    \begin{enumerate}
        \item the transition matrix $Q$ of the hidden states is doubly stochastic and $\|Q-\operatorname{C}_k\|_{\max}\leq k^{-c_2}$; 
        \item the emission probabilities are fixed as $\P(X = 1|Z = i) = p_i$ for all $i=1,2,\cdots,k$, 
    \end{enumerate}
   then there exists an estimator 
   $\widehat{Q}_k\left(X^n\right)$ such that
   \begin{equation}\label{eq: lem 4.1 est}
       \mathbb{P}_{X^n | Q}\left[\left\|\widehat{Q}_k\left(X^n\right)-Q\right\|_{\mathrm{F}}^2 \leq n^{-c_3}\right] \geq 0.99.
   \end{equation}
\end{lemma}


The constraints imposed on $Q$ in Lemma \ref{prop: redundancy lower bound} still result in a sufficiently rich model: one can show that there is a prior distribution $\mu$ supported on this set such that $h(Q) \gtrsim -k^2\log k$ for $Q\sim \mu$, where $h(\cdot)$ is the differential entropy (Lemma~\ref{lem: high ent construction}). For $Q\sim \mu$, a direct consequence of Lemma \ref{prop: redundancy lower bound} is $h(Q|\widetilde{Q}_k(X^n)) \lesssim -k^2\log n$ (Lemma~\ref{lem: red partial recovery}). Therefore, under this prior $\mu$ we have
\begin{align}\label{eq: high ent}
I(Q; X^n) = h(Q) - h(Q|X^n) \ge h(Q) - h(Q|\widetilde{Q}_k(X^n)) \gtrsim k^2\log n
\end{align}
as long as $n\ge k^D$ for a large constant $D>0$. 
The above two cases lead to the following theorem. 

\begin{theorem}[Lower bound on the redundancy]\label{thm: redd lower discr}
    There exist universal constants $c_0, c_1, c_2, D>0$ such that
    \begin{equation*}
        \red(\PHMM_n(k,\ell)) \ge c_0\left( k^2\log \frac{n}{k^2}+k\l\log \frac{n}{k\l}\right),
    \end{equation*}
    if either $k\ge \max\{\ell, c_1\}$ and $n>k^D$, or $n\geq c_2k\l$ and $\l\geq k$.
\end{theorem}
Combined with the reduction in \Cref{sec:reduction}, we conclude the proof of the lower bounds.

\printbibliography
%
\appendix
\section{Preliminaries and technical lemmas}\label{appdx: lemmas}
Recall the following information-theoretic quantities.
For probability distributions $P_X$ and $Q_X$ on the space $\calX$, 
the KL divergence is $\kl(P_X\|Q_X) = \Expect_P\left[ \log \frac{\mathrm{d}P_X}{\mathrm{d}Q_X}\right]$ if $P_X\ll Q_X$
and $\infty$ otherwise.
 The conditional KL divergence is 
$\kl(P_{X|Y}\|Q_{X|Y}|P_Y) = 
\Expect_{y\sim P_Y}[\kl(P_{X|Y=y}\|Q_{X|Y=y})]$. 
The mutual information between random variables $X$ and $Y$ is 
defined as $I(X;Y) = 
\kl(P_{XY}\|P_{X}\otimes P_Y) =\kl(P_{X|Y}\|P_{X}|P_Y) $
and 
the conditional mutual information is defined similarly $I(X;Y|Z) = \kl(P_{X|Y,Z}\|P_{X|Z}| P_{Y,Z})$.

We use $o, O, \omega, \Omega, \Theta$ following the common big-O notations, where an added $\widetilde{(\cdot)}$ denotes ignoring $\log$ factors. We also use $\lesssim, \gtrsim, \asymp$ to denote comparison ignoring universal constants. We shorthand $\naturals=\{1,2,\dots\}$ and $[t]=\{1,2,\dots, t\}$.

\subsection{Comparison with the formulation in \cite{sharan2018prediction}}\label{appdx: losses}
Recall that our prediction risk 
with respect to a true model $P_{X^{n+1}}$ is defined as
\begin{equation}
\risk(Q_{X_{n+1}|X^n};P_{X^{n+1}})=\E_{P_{X^{n+1}}}
[\kl(P_{X_{n+1}|X^n} \| Q_{X_{n+1}|X^n})]
    \label{eq:riskapp1}
\end{equation}
which compares a prediction algorithm $Q_{X_{n+1}|X^n}$ to the oracle  prediction $P_{X^{n+1}|X^n}$. 
Maximizing $P_{X^{n+1}}$ in a given model class, e.g., HMM, 
 leads to the 
 worst-case risk in \eqref{eq:risk}.

In \cite{sharan2018prediction}, the authors formulated the prediction problem differently as follows. First, it is assumed that observed sample path can be extended to a double-sided process $(X_t)_{t \in \integers}$. Then, for a sequence of predictors 
$Q_{X_{t+1}|X^t}$ indexed by the sample size $t=1,\ldots,n$
they consider the prediction loss in TV with respect to the respective oracle $P_{X^{t+1}|X^t_{-\infty}}$ and define:
\begin{equation}
\risk'(\{Q_{X_{t+1}|X^t}\}_{t=1}^n;P_{X^{n+1}}) 
=
\E_{X_{-\infty}^{n+1}} \qth{\frac{1}{n}  \sum_{t=1}^n \tv(P_{X_{t+1}|X_{-\infty}^t}, Q_{X_{t+1}|X^t})}.
    \label{eq:sharan}
\end{equation}
It is not straightforward to compare results under this formulation to our results partly due to this averaging over $t$, which means the prediction guarantee is not made for a given sample size $n$ by on average for a random sample size drawn uniformly from $1$ to $n$.
Nevertheless, a firm comparison one can make is the following. In the spirit of \prettyref{eq:sharan}, consider the following variant of \prettyref{eq:riskapp1}:
\begin{equation}
\risk''(Q_{X_{n+1}|X^n};P_{X^{n+1}})=\E_{P_{X^{n+1}}}
[\kl(P_{X_{n+1}|X^n_{-\infty}} \| Q_{X_{n+1}|X^n})].
    \label{eq:riskapp2}
\end{equation}
In other words, the goal is to compete with an oracle who only knows the true model parameters but also has access to infinite historical data. While this appears to be a more difficult task, 
for HMM the difference of these two risks is in fact negligible. Indeed, by the chain rule, we have for any predictor $Q$ and any model $P$,
$$\risk'' -\risk=I(X_{-\infty}^0; X_{n+1}|X^n)\leq I(Z_1; X_{n+1}|X^n)$$which, for HMM with $k$ hidden states, is at most $\log k/n$, thanks to \eqref{eq: decay info}. As a result, all of our results proved in $\risk$ (which is more natural in our settings) apply  immediately to $\risk''$.

\subsection{Technical lemmas}
 \begin{lemma}\label{lem: red partial recovery} Let $U\to X\to \widehat{U}$ be a Markov chain with $U$ being a continuous random variable with a density function $f_U$ 
 taking values in  $[0,t]^d$, and $\|\widehat{U} - U\|_2^2 \le d\varepsilon^2$ with probability at least $0.99$. 
Let $h(U) \triangleq \int f_U(u)du \log \frac{1}{f_U(u)}$ denote the differential entropy of $U$.  
 Then 
 \begin{align*}
    I(U; X) \ge h(U) + d\log\frac{1}{\varepsilon\sqrt{2\pi e}} - 0.01d\log\frac{1}{t} - \log 2. 
 \end{align*}
 \end{lemma}
 \begin{proof}
     Let $E$ be the event that $\|\widehat{U} - U\|_2^2 \le d\varepsilon^2$. Then
     \begin{align*}
        I(U;X) \ge I(U; \widehat{U}) &= h(U) - h(U|\widehat{U}) \\
        &= h(U) - h(U|\widehat{U}, \one_E) - I(U; \one_E | \widehat{U}) \\
        &\ge h(U) - h(U|\widehat{U}, E)\P(E) - h(U|\widehat{U}, E^c)\P(E^c) - \log 2 \\
        &\stepa{\ge} h(U) - h(U-\widehat{U}|\widehat{U}, E)\P(E) - 0.01d\log\frac{1}{t} - \log 2\\
        &\ge h(U) - h(U-\widehat{U}|E) - 0.01d\log\frac{1}{t} - \log 2 \\
        &\stepb{\ge} h(U) +d\log\frac{1}{\varepsilon\sqrt{2\pi e}} - 0.01d\log\frac{1}{t} - \log 2, 
     \end{align*}
    where (a) and (b) apply the fact that the differential entropy is maximized by the uniform (resp.~Gaussian) distribution subject to a support (resp.~second moment) constraint. 
 \end{proof}
 

The following lemma bounds the change of the prediction risk when certain auxiliary information is observed.

\begin{lemma} \label{lem: partial_reveal}
For a generic prior on $\theta$, the model parameters, and an auxiliary random variable $U$, it holds that
\begin{align*}
&\inf_{\widehat{P}_{X_{n+1}|X^n, U}} \E[\kl(P_{X_{n+1}|X^n,\theta, U}\|\widehat{P}_{X_{n+1}|X^n, U})] \\
&\le \inf_{\widehat{P}_{X_{n+1}|X^n}} \E[\kl(P_{X_{n+1}|X^n,\theta}\|\widehat{P}_{X_{n+1}|X^n})] + I(U;X_{n+1}|X^n,\theta). 
\end{align*}
where the expectation is taken with respect to both the model parameters $\theta$ according to the prior and the observations $X^n,U$
\end{lemma}

\begin{proof}
By the mutual information representation of the prediction risk (cf. \cite[Appendix A]{HJW23}), the statement is equivalent to
\begin{align*}
I(\theta;X_{n+1}|X^n, U) \le I(\theta;X_{n+1}|X^n) + I(U;X_{n+1}|X^n, \theta). 
\end{align*}
This is obvious since
\begin{align*}
I(\theta;X_{n+1}|X^n, U) = I(\theta;X_{n+1}|X^n) + I(U;X_{n+1}|X^n, \theta) - I(U;X_{n+1}|X^n)
\end{align*}
by the chain rule of mutual information.
\end{proof}

\begin{lemma}[Mirsky's theorem, \cite{Mirsky1960SYMMETRICGF}]\label{lem: Mirsky}
    For matrices $A, E \in \mathbb{R}^{m \times k}$ with $m \geq k$, it holds that $$
 \sum_{i=1}^k\left(\sigma_i(A+E)-\sigma_i(A)\right)^2 \leq\|E\|_{\mathrm{F}}^2,  $$
 where $\sigma_i (i\in [k])$ are the sorted singular values.
 \end{lemma}

\section{Deferred proofs of upper bounds}
\subsection{Proof of Proposition~\ref{prop:riskred}}
It holds that
$$
\begin{aligned}
& \kl\left(P_{X_{n+1} | X^n} \| \widetilde{Q}_{X_{n+1} | X^n} | P_{X^n}\right) \\
& \stepa{=} \mathbb{E}\left[\kl\left(P_{X_{n+1} | X^n}\left(\cdot | X^n\right) \bigg\| \frac{1}{n} \sum_{t=1}^n Q_{X_{t+1} | X^t}\left(\cdot | X_{n-t+1}^n\right)\right)\right] \\
& \stepb{\le} \frac{1}{n} \sum_{t=1}^n \mathbb{E}\left[\log \frac{P_{X_{n+1} | X^n}\left(X_{n+1} | X^n\right)}{Q_{X_{t+1} | X^t}\left(X_{n+1} | X_{n-t+1}^n\right)}\right] \\
&= \frac{1}{n} \sum_{t=1}^n \mathbb{E}\left[\log \frac{P_{X_{n+1} | X_{n-t+1}^n}\left(X_{n+1} | X_{n-t+1}^n\right)}{Q_{X_{t+1} | X^t}\left(X_{n+1} | X_{n-t+1}^n\right)}\right]+\mathbb{E}\left[\log \frac{P_{X_{n+1} | X^n}\left(X_{n+1} | X^n\right)}{P_{X_{n+1} | X_{n-t+1}^n}\left(X_{n+1} | X_{n-t+1}^n\right)}\right] \\
&\stepc{=} \frac{1}{n} \sum_{t=1}^n \underbrace{\mathbb{E}\left[\log \frac{P_{X_{t+1} | X^t}\left(X_{t+1} | X^t\right)}{Q_{X_{t+1} | X^t}\left(X_{t+1} | X^t\right)}\right]}_{\kl\left(P_{X_{t+1} | X^t} \| Q_{X_{t+1} | X^t} | P_{X^t}\right)}+I\left(X_{n+1} ; X^{n-t} | X_{n-t+1}^n\right) \\
&\stepd{=} \frac{1}{n}\left(\kl\left(P_{X^{n+1}} \| Q_{X^{n+1}}\right)-\kl\left(P_{X_1} \| Q_{X_1}\right)+\sum_{t=1}^n I\left(X_{n+1} ; X^{n-t} | X_{n-t+1}^n\right)\right)
\end{aligned}
$$
where (a) applies the definition \eqref{eq:Q1}; (b) is due to the convexity of the KL divergence; (c) uses the crucial fact that due to stationarity, $\left(X_{n-t+1}, \ldots, X_{n+1}\right) \stackrel{\text { law }}{=}\left(X_1, \ldots, X_{t+1}\right)$ and $P_{X_{n+1} | X_{n-t+1}^n}=$ $P_{X_{t+1} | X^t}$ for all $t$; (d) applies the chain rule of KL divergence. Using $\kl\left(P_{X_1} \| Q_{X_1}\right)\ge 0$ and taking the supremum over $P_{X^{n+1}}$, the proposition follows.  \qed 

\subsection{Proof of Proposition~\ref{prop: decay info}}
 By the hidden Markov structure, $X^{n-t} \to (X_{n-t+1}^n,  Z_{n-t+1}) \to X_{n+1}$ forms a Markov chain, or equivalently, 
 $X^{n-t} \to  Z_{n-t+1} \to X_{n+1}$ conditioned on $X_{n-t+1}^n$. Thus data processing inequality yields
$$I(X^{n-t}; X_{n+1}|X_{n-t+1}^n) \leq I(Z_{n-t+1}; X_{n+1}|X_{n-t+1}^n).$$
By stationarity, we have  $I(Z_{n-t+1}; X_{n+1}|X_{n-t+1}^n)=I(Z_1; X_{t+1}|X_1^t)$ and, furthermore
$$
   \sum_{t=1}^n I(X^{n-t}; X_{n+1}|X_{n-t+1}^n) \leq \sum_{t=1}^n I(Z_1; X_{t+1}|X_1^t)= I(Z_1; X^{n+1}),$$
   by the chain rule of mutual information. \qed 

   

Furthermore, since $(Z_1, X_2)\to Z_2\to X_{3}^{n+1}$ is a Markov chain,
     one has that:
     $$H(X_{n+1}|Z_1, X_2^n)\geq H(X_{n+1}|Z_2, X_3^n)=H(X_n|Z_1, X_2^{n-1})$$
     where the last equality follows from stationarity. Moreover, clearly $$H(X_{n+1}|X_1^n)\leq H(X_{n+1}|X_2^n)=H(X_n|X_1^{n-1}).$$
     Taking the difference of the two equations, one has that:
     $I(Z_1; X_{n+1}|X_1^n)\leq  I(Z_1; X_{n}|X_1^{n-1}).$ Therefore applying the chain rule $\sum_{t=1}^n I(Z_1; X_{t+1}|X_1^t)= I(Z_1; X^{n+1})$ and 
     $I(Z_1; X^{n+1})\leq H(Z_1)\leq  \log k$, we obtain
     \begin{equation}\label{eq: decay info}
         I(Z_1; X_{n+1}|X_1^n)\leq \frac{1}{n} I(Z_1; X^{n+1})\leq\frac{1}{n} \log k.
     \end{equation}
\subsection{Proof of Proposition~\ref{prop: red upper bound}}
The following result on Markov estimators was proven in \cite[Lemma 7]{han2021optimal}:  
\begin{lemma}[Markov redundancy]\label{lem:markov_redundancy}
    For $n\ge 2k^2$, any initial distribution $\pi$, and Markov transition matrix $M$, the marginal distribution $Q_{Z^{n+1}}$ defined in \eqref{eq:Q_joint_distribution} and \eqref{eq: add-one M} satisfies
    \begin{equation*}
        \max_{z^{n+1}} \log \frac{\pi\left(z_1\right) \prod_{t=1}^{n} M\left(z_{t+1}|z_t\right)}{Q_{Z^{n+1}}\left(z^{n+1}\right)}\lesssim k^2\log\frac{n}{k^2}.
    \end{equation*}
\end{lemma}
Now for the joint distribution $Q_{X^{n+1},Z^{n+1}}$, 
    $$\log\frac{\pi(z_1)\prod_{t=1}^{n} M(z_{t+1}|z_t)\prod_{t=1}^{n+1} T(x_t|z_t)}{Q_{X^{n+1},Z^{n+1}}(x^{n+1},z^{n+1})}=\log\frac{\pi(z_1)\prod_{t=1}^{n} M(z_{t+1}|z_t)}{Q_{Z^{n+1}}(z^{n+1})}+\log\frac{\prod_{t=1}^{n+1} T(x_t|z_t)}{ Q_{X^{n+1}|Z^{n+1}}(x^{n+1}|z^{n+1})} .$$
    The upper bound of the first term is stated in Lemma \ref{lem:markov_redundancy}. For the second term, let $N_i$ be the number of appearances of $i$ in $z^{n+1}$, and $N_{ij}$ be the number of appearances of $(i, j)$ in the state-emission pairs $(z_t,x_t)_{t=1}^{n+1}$. Then \eqref{eq:Q_joint_distribution} and \eqref{eq: add-one T} imply that
    $$Q_{X^{n+1}|Z^{n+1}}(x^{n+1}|z^{n+1})=\prod_{i=1}^k\frac{\prod_{j=1}^\l N_{ij}!}{\l^{\overline{N_i}}},$$
    where $a^{\overline{m}}=a(a+1)\cdots(a+m-1)$ is the rising factorial. Therefore, 
    \begin{align*}
        \log\frac{\prod_{t=1}^{n+1} T(x_t|z_t)}{Q_{X^{n+1}|Z^{n+1}}(x^{n+1}|z^{n+1})} &= \log\prod_{i=1}^k \l^{\overline{N_i}} \prod_{j=1}^{\ell}\frac{ T(j|i)^{N_{ij}}}{N_{ij}! }\\
        &\stepa{\le} \sum_{i=1}^k\log\frac{\l^{\overline{N_i}}}{N_i!}\\
        &\stepb{\le} \sum_{i=1}^k\left((\l-1) \log \left(1+\frac{N_i}{\l-1}\right)+N_i \log \left(1+\frac{\l-1}{N_i}\right)\right)\\
        &\stepc{\lesssim} k\l\log\frac{n}{k\l}, 
    \end{align*}
where (a) follows from the multinomial theorem $\binom{N_i}{N_{i1}\ \cdots \ N_{ik}}\prod_{j=1}^{\ell} T(j|i)^{N_{ij}}\le 1$, (b) is due to
\begin{align*}
\log \frac{\ell^{\overline{m}}}{m!} = \sum_{i=1}^{m} \log\left(1+\frac{\ell-1}{i}\right) &\le \int_0^m \log\left(1+\frac{\ell-1}{x}\right)\mathrm{d}x \\
&= (\ell-1)\log\left(1+\frac{m}{\ell-1}\right) + m\log\left(1+\frac{\ell-1}{m}\right),
\end{align*}
and (c) uses the concavity of $x\to \log x$, $\log(1+x)\le x$, and $\sum_{i=1}^k N_i = n+1$. The above two terms then complete the proof of the first statement. The second statement (marginalization) simply follows from
\begin{align*}
\red(Q_{X^{n+1}}) &\le \max_{x^{n+1}} \log\frac{\sum_{z^{n+1}} \pi(z_1)\prod_{t=1}^{n} M(z_{t+1}|z_t)\prod_{t=1}^{n+1} T(x_t|z_t)}{\sum_{z^{n+1}} Q_{X^{n+1},Z^{n+1}}(x^{n+1},z^{n+1})} \\
&\le \max_{x^{n+1},z^{n+1}} \log\frac{\pi(z_1)\prod_{t=1}^{n} M(z_{t+1}|z_t)\prod_{t=1}^{n+1} T(x_t|z_t)}{Q_{X^{n+1},Z^{n+1}}(x^{n+1},z^{n+1})} \\
&\lesssim k^2\log\frac{n}{k^2} + k\ell \log\frac{n}{k\ell}. 
\end{align*} \qed 

\subsection{Proof of Corollary \ref{cor: any emission}}
\label{sec:any emission}
By Propositions \ref{prop:riskred} and \ref{prop: decay info}, it suffices to construct $Q_{X^{n+1}}$ such that
\begin{align*}
\red(Q_{X^{n+1}}; \PHMM_n(k,\calQ)) \lesssim k^2\log\frac{n}{k^2} + kR({(n+1)/k}). 
\end{align*}
To this end, let $Q_{X^{n+1}}$ be the marginal distribution of $Q_{X^{n+1},Z^{n+1}}$, where $Q_{Z^{n+1}}(z^{n+1})$ is again given by \eqref{eq:Q_joint_distribution} and \eqref{eq: add-one M}. For the conditional distribution $Q_{X^{n+1}|Z^{n+1}}(x^{n+1}|z^{n+1})$, let $I_z(z^{n+1}) = \{t\in [n+1]: z_t = z\}$ be the time indices where $z_t = z$, we construct
\begin{align*}
Q_{X^{n+1}|Z^{n+1}}(x^{n+1}|z^{n+1}) = \prod_{z\in [k]} Q_{|I_z(z^{n+1})|}^\star(x_{I_z(z^{n+1})}), 
\end{align*}
where $Q^\star_m$ is the joint distribution such that $\red(Q^\star_m; \calQ^{\otimes m})\le R(m)$. Then for any HMM in $\PHMM_n(k,\calQ)$ with transition matrix $M$, stationary distribution $\pi$, and emission distributions $(Q_z)_{z\in [k]}\subseteq \calQ$, one has
\begin{align*}
&\log \frac{\pi(z_1)\prod_{t=1}^n M(z_{t+1}|z_t) \prod_{t=1}^{n+1} Q_{z_t}(x_t) }{Q_{X^{n+1},Z^{n+1}}(x^{n+1},z^{n+1})} \\
&= \log \frac{\pi(z_1)\prod_{t=1}^n M(z_{t+1}|z_t) }{Q_{Z^{n+1}}(z^{n+1})} + \sum_{z\in [k]} \log \frac{\prod_{t\in I_z(z^{n+1})} Q_{z}(x_t) }{Q_{|I_z(z^{n+1})|}^\star(x_{I_z(z^{n+1})})}. 
\end{align*}
Taking expectation over $x^{n+1}$ conditioned on $z^{n+1}$ gives that
\begin{align*}
\E\left[ \log \frac{\prod_{t\in I_z(z^{n+1})} Q_{z}(x_t) }{Q_{|I_z(z^{n+1})|}^\star(x_{I_z(z^{n+1})})} \right] \le \red(Q^\star_{|I_z(z^{n+1})|}; \calQ^{\otimes |I_z(z^{n+1})|})\le R(|I_z(z^{n+1})|), 
\end{align*}
so the concavity of $R$ leads to
\begin{align*}
\E\left[ \sum_{z\in [k]}\log \frac{\prod_{t\in I_z(z^{n+1})} Q_{z}(x_t) }{Q_{|I_z(z^{n+1})|}(x_{I_z(z^{n+1})})} \right] \le \sum_{z\in [k]} R({|I_z(z^{n+1})|}) \le kR({(n+1)/k}). 
\end{align*}
A combination of the above inequality and Lemma \ref{lem:markov_redundancy} completes the proof of the desired redundancy upper bound. 

\begin{remark}    
As discussed after \prettyref{cor: any emission}, for Gaussian location model $\calQ=\{\N(w, I_d): w\in [-1, 1]^d\}$, one has\footnote{For sharp asymptotic bounds of both average and pointwise redundancy in Gaussian models, see \cite{xie2000asymptotic}.} $\Red(\calQ^{\otimes n})\lesssim d\log n$ and thus 
$\Red(\PHMM_n(k,\calQ))\lesssim k^2 \log\frac{n}{k^2} +  d\log n$.
This is also implied by the pointwise redundancy bound of HMM with Gaussian emissions in \cite[Sec.~4.2.3]{gassiat2018universal}.
\end{remark}

\section{Deferred proofs of lower bounds}\label{appdx: lower HMM}

\subsection{Improved redundancy-based risk lower bound}
\label{sec:riskred-improved}

In this section we prove \prettyref{prop:risk_lb_general} which improves over \prettyref{prop:risk_lower_bound}. 
Let $\calQ$ be an arbitrary collection of distributions on the observation space $\calX$. Let $\PHMM_{n,\mathsf{U}}(k,\calQ)$ denote the class of all stationary HMMs of length $n$ with hidden states in $[k]$, emissions in $\calX$ with emission probabilities chosen in $\calQ$. In addition, let $\PHMM_{n,\mathsf{U}}(k,\calQ)$ denote the subclass of HMMs whose stationary distribution over the states is uniform.
Let $\red(\calQ)$ denote the redundancy of the distribution class $\calQ$ in the same sense as \prettyref{eq:red}, namely,
\[
\red(\calQ) \triangleq \inf_{Q'}\sup_{Q\in \calQ} \kl(Q\|Q')
\]
which, by the capacity-redundancy theorem (see, e.g., \cite[(13.10)]{PW-it}), equals
\begin{equation}
\red(\calQ) = \sup I(\theta;X)
\label{eq:red-cap}
\end{equation}
where $P_{X|\theta}\in\calQ$ and the maximization is taken over the distribution (prior) of $\theta$.
\prettyref{prop:risk_lb_general} lower bounds the prediction risk of 
$\PHMM_{n}(k,\calQ)$ using the redundancy of the (slightly smaller) class $\PHMM_{n,\mathsf{U}}(k,\calQ)$, which is the set of all stationary HMMs in $\PHMM_{n}(k,\calQ)$ with uniform stationary distribution for the hidden states.  

\vspace{6pt}
\begin{proof}[Proof of Proposition~\ref{prop:risk_lb_general}] Given a HMM configuration $Q$ 
(comprising the transition probability $Q_{X_2|X_1}$ and emission probability $Q_{X_1|Z_1}$) in $\PHMM_{n+1,\mathsf{U}}(k-1,\calQ)$, we construct an HMM $P$ in $\PHMM_n(k,\calQ)$ as follows.
 in a similar way as that in 
 Section \ref{sec:reduction}.
The emission probabilities for states  $i\in[k-1]$ are identical, namely $P_{X_1|Z_1=i}=Q_{X_1|Z_1=i}$ and 
$P_{X_1|Z_1=k}$ is any fixed distribution in $\calQ$.
The transition probabilities 
$P_{X_2|X_1}$ is defined based on 
$Q_{X_2|X_1}$ using \prettyref{eq:reduction_transition}.
Hence, both HMM configurations $P$ and $Q$ are parameterized  by the transition matrix $Q_{X_2|X_1}$ and the emission probability matrix $Q_{X_1|Z_1}$. Let $\theta$ denote the collection of these parameters.


Let $T\in \{1,2,\cdots,n, \bot\}$ be the smallest $t\in [n]$ such that $Z_t\neq k$ (if $Z^n = k^n$ then $T = \bot$). In the later proof we will condition on $T$; a subtlety here is that $T$ is determined by the states $Z^n$ but it may not be measurable with respect to $X^n$. Nevertheless, 
we will consider the setting where both the estimand  $P_{X_{n+1}|X^n}$ and the estimator $\widehat{P}_{X_{n+1}|X^n}$ have access to the extra information $T$. While 
replacing 
$\widehat{P}_{X_{n+1}|X^n}$  by $\widehat{P}_{X_{n+1}|X^n,T}$ 
is valid for the sake of lower bound, replacing 
$P_{X_{n+1}|X^n}$ by $P_{X_{n+1}|X^n,T}$ requires further justification. By Lemma~\ref{lem: partial_reveal}, this increases the prediction risk by at most $I(T; X_{n+1}|X^n)$ which we bound below.

\begin{lemma}\label{lem:remainder}
For any fixed $\theta$, we have
\begin{align*}
I(T; X_{n+1} | X^n) \lesssim \frac{\log(kn)}{n}. 
\end{align*}
\end{lemma}
\begin{proof}
Let $U_t = \one_{T=t}$ for $t\in [n]$, then $T$ is determined by $U^n$. Then
\begin{align*}
I(T; X_{n+1}|X^n) &\le I(U^n; X_{n+1}|X^n) \\
&= \sum_{t=1}^n I(U_t; X_{n+1}|X^n, U^{t-1}) \\
&\stepa{\le} I(U_1; X_{n+1}|X^n) + \sum_{t=2}^n I(U_t; X_{n+1}|X^{n}, U^{t-1}=0^{t-1}) \\
&= I(U_1; X_{n+1}|X^n) +\sum_{t=2}^n I(U_t; X_{n+1}|X^{n}, Z^{t-1}=k^{t-1}) \\
&\stepb{=} I(U_1; X_{n+1}|X^n) +\sum_{t=2}^n I(U_t; X_{n+1}|X_t^{n}, Z_{t-1}=k) \\
&\stepc{=} I(U_1; X_{n+1}|X^n) +\sum_{t=2}^n I(U_2; X_{n+3-t} | X_2^{n+2-t}, Z_1 = k) \\
&=  I(U_1; X_{n+1}|X^n) + I(U_2; X_3^{n+1} | X_2, Z_1 = k) \\
&\stepd{\le} \frac{H(Z_1)}{n} + H(U_2 | Z_1 = k) \stepe{\lesssim}\frac{\log(nk)}{n}, 
\end{align*}
where (a) uses the fact that $U_t = 0$ deterministically if any entry of $U^{t-1}$ is one so that 
$$I(U_t; X_{n+1}|X^n, U^{t-1})=
I(U_t; X_{n+1}|X^n, U^{t-1}=0^{t-1})
\P(U^{t-1}=0^{t-1} | X^n);$$ 
(b) is due to the Markov structure $(X^{t-1}, Z^{t-2})\to (X_t^m, Z_{t-1})\to U_t$ given $Z^{t-1} = k^{t-1}$ for $m\in \{n, n+1\}$;
(c) is due to stationarity, and (d) follows from $I(U_1; X_{n+1}|X^n)\leq I(Z_1; X_{n+1}|X^n)\leq_{\eqref{eq: decay info}}{H(Z_1)}/{n}$;
(e) is because 
by definition \prettyref{eq:reduction_transition},
conditioned on $Z_1=k$, $U_2 \sim \Bern(\frac{1}{n})$.
\end{proof}

Next we assume an arbitrary prior on $\theta$ and use this Bayesian setting to lower bound the prediction risk.
Similar to the analysis in Section \ref{sec:reduction}, it is clear that for every $t\in [n]$, it holds that $\P(T=t) = \Omega(1/n)$. Conditioned on the event $T=t$, the random distribution $P_{X_{n+1}|X^n, \theta, T=t}$ shares the same law as $P_{Y_{n-t+2}|Y^{n-t+1}, \theta}$, where $Y^m$ is a sample path of length $m$ from the same $k$-state HMM, but with the initial hidden state $Z_1$ drawn uniformly from $[k-1]$. Similarly, conditioned on $T=t$, the posterior distribution of $\theta$ given $X^n$ has the same law as the posterior distribution of $\theta$ given $Y^{n-t+1}$. Consequently, the Bayes prediction risk satisfies
\begin{align*}
&\inf_{\widehat{P}_{X_{n+1}|X^n, T}} \E[\kl(P_{X_{n+1}|X^n,\theta,T}\|\widehat{P}_{X_{n+1}|X^n, T})] \\
&\ge \sum_{t=1}^n \P(T=t) \cdot \inf_{\widehat{P}_{X_{n+1}|X^n, T}} \E[\kl(P_{X_{n+1}|X^n,\theta,T}\|\widehat{P}_{X_{n+1}|X^n, T}) | T = t] \\
&= \sum_{t=1}^n \P(T=t) \cdot \inf_{\widehat{P}_{Y_{n+2-t}|Y^{n+1-t}} } \E[\kl(P_{Y_{n+2-t}|Y^{n+1-t}, \theta}\|\widehat{P}_{Y_{n+2-t}|Y^{n+1-t}})] \\
&\gtrsim \frac{1}{n} \sum_{t=1}^n I(\theta; Y_{n+2-t}|Y^{n+1-t}) = \frac{I(\theta; Y^{n+1}) - I(\theta;Y_1)}{n}. 
\end{align*}

To deal with the above terms, for the second mutual information we have
\begin{align*}
I(\theta; Y_1) \le I(\theta; Y_1, Z_1) = I(\theta; Y_1 | Z_1) + I(\theta; Z_1) \stepa{=} I(\theta; Y_1 | Z_1) \stepb{\le} \red(\calQ), 
\end{align*}
where (a) uses that the distribution of $Z_1$ is uniform regardless of $\theta$ by definition of the model class $\PHMM_{n+1,\mathsf{U}}$, and (b) is because for any state $i\in[k]$, $I(\theta; Y_1 | Z_1=i) \leq \Red(\calQ)$ by the capacity-redundancy representation \prettyref{eq:red-cap}, since the emission probabilities $P_{ Y_1 | Z_1=i,\theta}$ are chosen from $\calQ$.

For the first mutual information $I(\theta;Y^{n+1})$, let $(Y^{n+1}, Z^{n+1})$ be the observations and hidden states in the HMM with $k$ states (starting from $Z_1\sim \mathsf{Unif}([k-1])$), and $X_Q^{n+1}$ be the observations in the reduced HMM (with $k-1$ states and model parameters $Q$). Let $E$ be the event that $Z_t\neq k$ for all $t\in [n+1]$, by chain rule 
$I(\theta; Y^{n+1}) = I(\theta,\one_E;Y^{n+1})- I(\one_E;
Y^{n+1}|\theta)$, we have
\begin{align*}
I(\theta; Y^{n+1}) \ge I(\theta; Y^{n+1} | \one_E) - \log 2 \ge \P(E) I(\theta; Y^{n+1} | E) - \log 2. 
\end{align*}
We note that $\P(E) = (1-1/n)^n = \Omega(1)$, and the joint distribution $P_{\theta, Y^{n+1}|E}$ is the same as $P_{\theta, X_Q^{n+1}}$. Therefore, 
\begin{align*}
I(\theta; Y^{n+1}) \gtrsim I(\theta; X_Q^{n+1}) - \log 2.
\end{align*}
Using Lemmas \ref{lem: partial_reveal}, \ref{lem:remainder},
as well as the above inequalities yields a lower bound on the Bayes (and hence minimax) prediction risk:
\begin{align*}
\risk(\PHMM_n(k,\calQ)) \gtrsim \frac{1}{n}I(\theta; X_Q^{n+1}) - \frac{\log(nk) + \red(\calQ)}{n}.
\end{align*}
Finally, taking the supremum over the prior of the model parameter $\theta$ and invoking the capacity-redundancy identity \prettyref{eq:red-cap} complete the proof of Proposition \ref{prop:risk_lb_general}.
\end{proof}
\subsection{The case of $\ell=2$ and proof of Corollary~\ref{cor:general_emission_lb}}

When $\ell = 2$, we distinguish into two cases: $k=2$, and $k\ge C$ for a large absolute constant $C>0$. In the first case we establish an $\Omega(\log n/n)$ lower bound for the prediction risk of $\PHMM_n(2,2)$, and in the second case we prove Corollary \ref{cor:general_emission_lb} when $n\ge k^D$, which implies the lower bound in Theorem \ref{thm: stats bounds} for $\RiskHMM(n,k,2)$. 

\subsubsection{$k=2$}
Let us first consider the case $k=\l=2$. Note that we cannot directly apply Proposition~\ref{prop:risk_lb_general} as the remainder term $\log(nk)/n$ is too large. Instead, similar to the general reduction in Section \ref{sec:reduction}, we consider a specific transition matrix for the hidden states: 
 \begin{align*}
 M = \begin{bmatrix}
        \frac{n-1}{n}&\frac{1}{n} \\
         \frac{1}{n}&\frac{n-1}{n}     \end{bmatrix}. 
 \end{align*}
 For the emission probabilities, we set $\P(X=1|Z=0) = 1$ and $\P(X=1|Z=1) = \theta$, where the prior distribution of $\theta$ is $\mathsf{Unif}([0.1,0.9])$. We show that the Bayes prediction risk is $\Omega(\log n/n)$. 

 For $t\in [n-1]$, let $E_t$ be the event $Z^t = 0^t, Z_{t+1}^n = 1$ and $X_n = 0$. Clearly
 \begin{align*}
 \P(E_t) \ge \frac{1}{2}\cdot \left(1-\frac{1}{n}\right)^{t-1}\cdot \frac{1}{n}\cdot \left(1-\frac{1}{n}\right)^{n-t-1}\cdot (1-0.9) \ge \frac{1}{60n}.
 \end{align*}
 We investigate the Bayes prediction risk conditioned on $E_t$. For the true distribution $P_{X_{n+1}|X^n, \theta}$, the event $E_t$ implies $X_n =0$ (observable from $X^n$) and consequently $Z_n = 1$, so that $P_{X_{n+1}|X^n, \theta} = \mathsf{Bern}(\frac{n-1}{n}\theta+\frac{1}{n})$. For the estimator $Q_{X_{n+1}|X^n}$, the posterior distribution of $\theta$ given $(X^n, E_t)$ is the same as the posterior distribution of $\theta$ given $Y^{n-t}\sim \mathsf{Bern}(\theta)^{\otimes(n-t)}$. Consequently,
 \begin{align*}
 &\inf_{Q(\cdot|X^n)} \E\left[\log\frac{P(X_{n+1}|X^n, \theta)}{Q(X_{n+1}|X^n)}\bigg| E_t\right] \\
 &\ge \inf_{Q(\cdot|X^n, E_t)} \E\left[\log\frac{P(X_{n+1}|X^n, \theta)}{Q(X_{n+1}|X^n, E_t)}\bigg| E_t\right] \\
 &\ge \inf_{\widehat{p}(X^n, E_t)} \E \left[\kl\left(\mathsf{Bern}\left(\frac{n-1}{n}\theta+\frac{1}{n}\right) \| \mathsf{Bern}(\widehat{p}(X^n, E_t))\right)\bigg|E_t\right] \\
 &= \inf_{\widehat{p}(Y^{n-t})} \E \left[\kl\left(\mathsf{Bern}\left(\frac{n-1}{n}\theta+\frac{1}{n}\right) \| \mathsf{Bern}(\widehat{p}(Y^{n-t}))\right)\right] \\
 &\ge 2\left(\frac{n-1}{n}\right)^2 \inf_{\widehat{\theta}(Y^{n-t})} \E \left[(\theta-\widehat{\theta}(Y^{n-t}))^2\right] \gtrsim \frac{1}{n-t}, 
 \end{align*}
 where we have used Pinsker's inequality and the $\Omega(1/(n-t))$ Bayes mean squared error of estimating $\theta$ under $Y^{n-t}\sim \mathsf{Bern}(\theta)^{\otimes(n-t)}$. Therefore, the overall Bayes prediction risk is
 \begin{align*}
 \inf_{Q(\cdot|X^n)} \E\left[\log \frac{P(X_{n+1}|X^n,\theta)}{Q(X_{n+1}|X^n)}\right] &\ge \sum_{t=1}^{n-1} \P(E_t)\cdot \inf_{Q(\cdot|X^n)} \E\left[\log\frac{P(X_{n+1}|X^n, \theta)}{Q(X_{n+1}|X^n)}\bigg| E_t\right] \\
 &\gtrsim \sum_{t=1}^{n-1} \frac{1}{n}\cdot \frac{1}{n-t} \gtrsim \frac{\log n}{n}, 
 \end{align*}
 completing the proof of $\RiskHMM(n,2,2)=\Omega(\log n/n)$. 

\subsubsection{Corollary~\ref{cor:general_emission_lb}: $k\ge C$ and $n\ge k^D$}
Since Corollary~\ref{cor:general_emission_lb} implies the lower bound of \Cref{thm: stats bounds} in this case, it suffices to prove Corollary~\ref{cor:general_emission_lb}. In fact, by Proposition~\ref{prop:risk_lb_general}, the final step in proving Corollary~\ref{cor:general_emission_lb} is to lower bound the redundancy. This can be divided into two steps:
\begin{enumerate}
    \item That $\red(\PHMM_{n,\mathsf{U}}(k,\calQ))\gtrsim k^2\log n$. This follows immediately from Lemma~\ref{prop: redundancy lower bound} and the construction therein (Lemma~\ref{lem: high ent construction}). Indeed, let $Q$ the transition matrix be sampled according to Lemma~\ref{lem: high ent construction} and plug in $t_1<t_2$ in Lemma~\ref{prop: redundancy lower bound} be according to $c_1<c_2$ in Corollary~\ref{cor:general_emission_lb}. Let $f$ be the map pushing $Q\in \calQ$ to Bernoulli's in Corollary~\ref{cor:general_emission_lb}. Then one has that:
    $$I(X_{\calQ}^n; \theta)\geq I(f(X_\calQ^n); \theta)\gtrsim k^2\log n.$$
    \item That $\red(\PHMM_{n,\mathsf{U}}(k,\calQ))\gtrsim k\red(\calQ^{\otimes\lfloor n/k\rfloor})-\log k$, where we ignore the $k\to k-1$ issue since $\red(\calQ^t)$ grows at most linearly in $t$. To show this, consider the HMM whose latent states evolve according to the cycle $\T_k,$ and each emission $P_i=\P(X|Z=i)$ is sampled independently from some distribution $\mu$ supported on $\calQ$. Then:
    $$\sup_\mu I(P^k; X^n)\geq \sup_\mu I(P^k; X^n,Z_1)-\log k=k\red(\calQ^{\otimes\lfloor n/k\rfloor})-\log k$$ by optimizing the prior $\mu$. Therefore Corollary~\ref{cor:general_emission_lb} is proven.
\end{enumerate}
\begin{remark}
    In fact, the risk from emission $\frac{k}{n}\red(\calQ^{\otimes \lfloor n/k\rfloor})$ holds even for $k=2$ and $n\in O(1)$. The constraint on $k\geq C, n\geq k^D$ is only used to show $\risk\gtrsim \frac{k^2}{n}\log n$ via Lemma~\ref{prop: redundancy lower bound}.
\end{remark}

\subsection{Proof of Lemma \ref{prop: redundancy lower bound}}
The proof of Lemma \ref{prop: redundancy lower bound} relies critically on the following Markov property:
\begin{align*}
P_{X_{-L}^L}(x_{-L}^L) &= \sum_{z_0 \in [k]} P_{Z_0}(z_0) P_{X_{-L}^L|Z_0}(x_{-L}^L |z_0) \\
&= \sum_{z_0 \in [k]} P_{Z_0}(z_0) P_{X_0|Z_0}(x_0 | z_0) P_{X_1^L|Z_0}(x_1^L | z_0) P_{X_{-L}^{-1}|Z_0}(x_{-L}^{-1} | z_0) \\
&= \frac{1}{k}\sum_{z_0 \in [k]} P_{X_0|Z_0}(x_0 | z_0) P_{X_1^L|Z_0}(x_1^L | z_0) P_{X_{-L}^{-1}|Z_0}(x_{-L}^{-1} | z_0), 
\end{align*}
where the last step uses the fact that the stationary distribution is uniform when the transition matrix is doubly stochastic. In a tensor form, we write $P_{X_{-L}^L}$ as a $2^L\times 2^L\times 2$ tensor $M$, and express  $P_{X_0|Z_0=z_0}, P_{X_1^L|Z_0=z_0}, P_{X_{-L}^{-1}|Z_0=z_0}$ as vectors $o_{z_0}\in \mathbb{R}^2, e_{z_0}\in \mathbb{R}^{2^L}, f_{z_0}\in \mathbb{R}^{2^L}$, respectively. Then we have the following tensor decomposition of the moment matrix (\cite{huang2015minimal}): 
\begin{align}\label{eq:tensor_decomposition}
    M = \frac{1}{k}\sum_{z_0 \in [k]} e_{z_0}\otimes f_{z_0}\otimes o_{z_0}. 
\end{align}
We also let $E, F\in\mathbb{R}^{2^L\times k}$ be the matrices with column vectors $e_{z_0}$ and $f_{z_0}$, respectively. 

Based on the above tensor decomposition, the proof consists of several steps:
\begin{enumerate}
    \item In the first step, we show that there exist constants $C_1,C_2,d_1,d_2,d_3>0$ such that for $L = \lceil d_1\log k\rceil$, based on the HMM trajectory $X^n$ one may construct an estimator $\widehat{M}_n$ such that
    \begin{align}\label{eq:first_step}
    \P(\| \widehat{M}_n - M \|_{\mathrm{F}} \ge n^{-d_2}) \le C_1\exp(-C_2n^{d_3}). 
    \end{align}
    \item In the second step, we show that based on the estimate $\widehat{M}_n$ and tensor decomposition, one may construct estimators $\widehat{E}_n, \widehat{F}_n$ for matrices $E,F$ such that (up to permutations of columns)
    \begin{align}\label{eq:second_step}
    \P( \| \widehat{E}_n - E \|_{\mathrm{F}} + \| \widehat{F}_n - F \|_{\mathrm{F}} \ge n^{-d_4} ) \le 0.01, 
    \end{align}
    where $d_4>0$ is an absolute constant; 
    \item In the last step, we conclude the statement of Lemma \ref{prop: redundancy lower bound} based on \eqref{eq:second_step}. 
\end{enumerate}
We break these steps into several subsections. 

\subsubsection{First step: proof of \eqref{eq:first_step}}
Note that if the transition matrix is precisely the cycle $\operatorname{C}_k$, then $X^{2L+1}$ and $X_{T}^{T+2L}$ are independent if we choose a random time index $T\sim \mathsf{Unif}(\{2L+2, \cdots, 2L+k+1\})$. The following lemma states that this is essentially the case whenever the transition matrix $Q$ is close to $\operatorname{C}_k$. 

 \begin{lemma}\label{lem: sampling windows}
 Consider a Markov chain $(Z_t)$ with transition matrix $Q\in\mathbb{R}^{k\times k}$ such that $\|Q-\operatorname{C}_k\|_{\max}\leq \eps$, and that the stationary distribution is uniform. Then for any $m\in \naturals$, there exists a distribution $P_m$ supported on $[mk]$ such that if $T\sim P_m$ (independent of the chain), then
     $$\max_{z\in [k]}\| \mathsf{Unif}([k]) - P_{Z_T|Z_0=z}\|_{\mathrm{TV}} \le (2k\eps)^m.$$
 \end{lemma}

 \begin{proof}
    Let $T_1, T_2, \cdots$ be i.i.d.~from $\mathsf{Unif}([k])$, and $T = \sum_{i=1}^m T_i$. Clearly $T$ is supported on $[mk]$. Note that for any $z, z'\in [k]$, one has
     $$\P(Z_{T_1}=z'|Z_0=z)\geq \frac{1}{k}(1-\eps)^k,$$
     where $1/k$ is the probability that the $T_1$ equals to the time $z$ travels to $z'$ under $\operatorname{C}_k$, and $(1-\eps)^k$ lower bounds the probability of following the path in $\operatorname{C}_k$. This implies that
     \begin{align*}
        \max_{z\in [k]} \| \mathsf{Unif}([k]) - P_{Z_{T_1}|Z_0=z}\|_{\mathrm{TV}} \le 1 - (1-\varepsilon)^k \le k\varepsilon. 
     \end{align*}
    Now the result follows from the standard mixing bound for the new chain $(Z_0,Z_{T_1},Z_{T_1+T_2},\cdots)$ (cf. \cite[Lemma 4.12]{levin2017markov}). 
 \end{proof}

Now we choose $m = \lceil \sqrt{n} \rceil$, and consider $J = \lceil n^{1/3} \rceil$ disjoint intervals $I_j = [(j-1)(mk+2L) + T_j, (j-1)(mk+2L) + T_j + 2L] \subseteq [(j-1)(mk+2L)+1, j(mk+2L)]$, where $T_1,\cdots,T_J$ are i.i.d.~according to the distribution $P_m$ in Lemma \ref{lem: sampling windows}. Note that $I_1,\cdots,I_J\subseteq [n]$ for $n\ge k^6$. By the HMM structure, for any $j\in [J]$ it holds that
\begin{align*}
\E\| P_{X^{2L+1}} - P_{X_{I_j} | (X_{I_i})_{i<j} } \|_{\mathrm{TV}} &\stepa{\le}\E\| P_{Z_1} - P_{Z_{t_{j-1} + T_j}|(X_{I_i})_{i<j}} \|_{\mathrm{TV}} \\
&\stepb{\le} \E\| \mathsf{Unif}([k]) - P_{Z_{t_{j-1} + T_j}|Z_{t_{j-1}}} \|_{\mathrm{TV}} \\
&\stepc{\le} (2k\varepsilon)^m, 
\end{align*}
where (a) follows from the data processing inequality for the TV distance (where $t_j$ is the end time of $I_j$), (b) is due to the Markov structure $(X_{I_i})_{i<j}\to Z_{t_{j-1}} \to Z_{t_{j-1}+T_j}$, and (c) follows from Lemma \ref{lem: sampling windows}. Therefore, by the subadditivity of TV distance,
\begin{align}\label{eq:TV_upper_bound}
\|P_{X^{2L+1}}^{\otimes J} - P_{(X_{I_j})_{j\in J}}\|_{\mathrm{TV}} \le \sum_{j=1}^J \E\| P_{X^{2L+1}} - P_{X_{I_j} | (X_{I_i})_{i<j} } \|_{\mathrm{TV}} \le J(2k\varepsilon)^m = e^{-\widetilde{\Omega}(n)},
\end{align}
where the last inequality follows from the assumption $\varepsilon\le k^{-c_2}$ in Lemma \ref{prop: redundancy lower bound}, for $c_2\ge 2$. 

Using the near independence, we proceed to estimate the joint distribution $P_{X^{2L+1}}$ (or equivalently the tensor $M$) based on the empirical distribution $\widehat{M}_n$ of $\{X_{I_j}\}_{j\in J}$. If $\{X_{I_j}\}_{j\in J}$ were indeed i.i.d., by Hoeffding's inequality and union bound we would have
\begin{align*}
\P( \| \widehat{M}_n - M\|_{\mathrm{F}} \ge n^{-d_2} ) &\le \P( \| \widehat{M}_n - M\|_{\max} \ge n^{-d_2} ) \\
&\le 2^{2L+1}\cdot 2\exp(-2J(n^{-d_2})^2) \\
&= \exp(-\widetilde{\Omega}(n^{1/2 - 2d_2})).
\end{align*}
For weakly dependent $\{X_{I_j}\}_{j\in J}$, we invoke \eqref{eq:TV_upper_bound} to conclude that
\begin{align*}
\P( \| \widehat{M}_n - M\|_{\mathrm{F}} \ge n^{-d_2} ) \le \exp(-\widetilde{\Omega}(n^{1/2 - 2d_2})) + e^{-\widetilde{\Omega}(n)}. 
\end{align*}
Consequently, \eqref{eq:first_step} holds with $d_2 = d_3 = 1/8$. 

\subsubsection{Second step: proof of \eqref{eq:second_step}}
Given the estimate $\widehat{M}_n$ of the tensor $M$, we aim to recover the matrices $E$ and $F$ (up to permutations of columns). To this end we recall the following result in tensor decomposition. 

 \begin{lemma}[Stability of tensor decomposition, Theorem 2.3 in \cite{Bhaskara2013Smoothed}]\label{lem: tensor decomp}
Let $T=\sum_{i=1}^k u_i \otimes v_i \otimes w_i$ be a tensor satisfying the following conditions:
\begin{enumerate}
     \item The condition numbers $\kappa(U), \kappa(V) \leq \kappa$, where $U, V$ are $m\times k$ matrices with column vectors $u_i$ and $v_i$, respectively; 
     \item The vectors $w_i\in \mathbb{R}^2$ are not close to parallel: $\min_{i \neq j}\left\|\frac{w_i}{\left\|w_i\right\|}-\frac{w_j}{\left\|w_j\right\|}\right\|_2 \geq \delta>0$;
     \item The decompositions are bounded: for all $i,\left\|u_i\right\|_2,\left\|v_i\right\|_2,\left\|w_i\right\|_2 \leq 1$.
 \end{enumerate}
Now given the noisy tensor $T+E \in \mathbb{R}^{m \times m \times 2}$ with the entries of $E$ bounded by $\eps$, there exists an efficient algorithm that returns each rank one term in the decomposition of $T$ (up to renaming), within an additive error of $\eps\cdot \operatorname{poly}(m, k, \kappa, 1 / \delta)$.
 \end{lemma}

Note that in our tensor decomposition \eqref{eq:tensor_decomposition}, we have $m = 2^{L} = \operatorname{poly}(k)$, and $\varepsilon \le kn^{-d_2}$ by \eqref{eq:first_step} with high probability. Consequently, if the conditions of Lemma \ref{lem: tensor decomp} hold with $\kappa, 1/\delta = \operatorname{poly}(k)$, then by choosing $c_1>0$ large enough in the condition $n\ge k^{c_1}$ in Lemma \ref{prop: redundancy lower bound}, Lemma \ref{lem: tensor decomp} will imply \eqref{eq:second_step}. Hence it remains to verify the conditions of Lemma \ref{lem: tensor decomp}. 

The third condition is straightforward: a probability vector $e$ must satisfy $\|e\|_2\le \|e\|_1 = 1$. To verify the first two conditions, we use a probabilistic argument and choose the vectors $o_i$ (i.e. emission probabilities $P_{X|Z}$) randomly. Specifically, for fixed constants $a_1,a_2$ with $0<a_1<a_2<\pi/2$, we generate i.i.d.~angles $\theta_1, \cdots, \theta_k\sim \mathsf{Unif}([a_1,a_2])$, and set
\begin{align*}
    p_i = \P(X = 1 | Z = i) = \frac{\cos\theta_i}{\cos\theta_i + \sin\theta_i}, \quad i\in [k]. 
\end{align*}
Note that for appropriately chosen constants $a_1,a_2$, the condition $p_i\in (t_1,t_2)$ required in Lemma \ref{prop: redundancy lower bound} holds almost surely. The reason behind the choice of $p_i$ is summarized by the following lemma.

\begin{lemma}\label{lem:singular_value_LB}
If $Q = \operatorname{C}_k$ and $k\ge L$, then for a large enough constant $d_1>0$, there exist absolute constants $d_5, d_6>0$ such that
\begin{align*}
\P( \max\{\kappa(E), \kappa(F)\} \le k^{d_5}) \ge 1 - 2k^{-d_6}, 
\end{align*}
where $\kappa(E)$ denotes the condition number of $E$, i.e. the ratio between the largest and the $k$th singular values of $E$. 
\end{lemma}
\begin{proof}
By symmetry we only prove the claim for $E$. If $Q = \operatorname{C}_k$ and $k\ge L$, we have
\begin{align*}
E_{x^L, i} = P_{X_1^L|Z_0}(x^L|i) = \prod_{t=1}^{L} \frac{(\cos\theta_{t+i})^{1-x_t} (\sin\theta_{t+i})^{x_t} }{\cos \theta_{t+i} + \sin\theta_{t+i}}, 
\end{align*}
where the indices of $\theta$ are understood modulo $k$. Consequently, for $i,j\in [k]$, 
\begin{align*}
(E^\top E)_{i,j} = \prod_{t=1}^L \frac{\cos(\theta_{t+i} - \theta_{t+j})}{(\cos\theta_{t+i}+\sin\theta_{t+i})(\cos\theta_{t+j} + \sin\theta_{t+j})}. 
\end{align*}
In matrix forms, $E^\top E = DAD$, where $D\in\mathbb{R}^{k\times k}$ is a diagonal matrix with $D_{ii} = \prod_{t=1}^L (\cos\theta_{t+i}+\sin\theta_{t+i})^{-1} \in [2^{-L/2},1]$, and $A_{ij} = \prod_{t=1}^L \cos(\theta_{t+i} - \theta_{t+j})$. Consequently, 
\begin{align}\label{eq:reduction_to_A}
    \lambda_1(E^\top E)\le \lambda_1(A), \qquad \lambda_k(E^\top E) \ge (2^{-L/2})^2 \lambda_k(A)\ge 2^{-L}\lambda_k(A). 
\end{align}
Next we analyze the matrix $A$. Clearly the diagonal entries of $A$ are all $1$. For off-diagonal entries $A_{ij}$, as $k\ge L$, we may pick $N\ge L/2$ elements $t_1,\cdots,t_N\in [L]$ such that the set $\{i+t_r, j+t_r: r\in [N]\}$ contains $2N$ distinct elements. Then
\begin{align*}
\log A_{ij}\le \sum_{r=1}^N \log\cos(\theta_{i+t_r} - \theta_{j+t_r}), 
\end{align*}
where the RHS is the sum of $N$ i.i.d.~random variables. Since $\E[\log\cos(\theta_1 - \theta_2)]\le -c(a_1,a_2)$ for some constant $c(a_1,a_2)>0$, and $|\log\cos(\theta_1-\theta_2)|\le |\log\cos(a_1 - a_2)|$ almost surely, Hoeffding's inequality implies that $\log A_{ij}= -\Omega(N)$ with probability at least $1-\exp(-\Omega(N))=1-k^{-\Omega(d_1)}$. By choosing $d_1$ large enough, a union bound implies that with probability at least $1-k^{-d_6}$, all off-diagonal entries of $A$ have magnitude at most $1/(2k)$. Then by Gershgorin circle theorem, 
\begin{align*}
    \lambda_1(A) &\le 1 + \max_{i\in [k]}\sum_{j\neq i}|A_{ij}| \le 1 + \frac{1}{2k}\cdot (k-1) \le \frac{3}{2}, \\ 
    \lambda_k(A) &\ge 1 - \max_{i\in [k]}\sum_{j\neq i}|A_{ij}| \ge 1 - \frac{1}{2k}\cdot (k-1) \ge \frac{1}{2}
\end{align*}
hold with high probability. A combination of \eqref{eq:reduction_to_A} and the above result completes the proof. 
\end{proof}

The result of Lemma \ref{lem:singular_value_LB} assumes $Q = \operatorname{C}_k$, and we still need to generalize it to the case where $\|Q-\operatorname{C}_k\|_{\max} \le k^{-c_2}$. To this end we apply a matrix perturbation analysis. Let $E^\star$ be the matrix $E$ with columns $e_i^\star$ under $Q = \operatorname{C}_k$, then 
\begin{align}
\|e_i^\star - e_i\|_1 &= 2\| P_{X^L|Z_0=i, Q=\operatorname{C}_k} - P_{X^L|Z_0=i, Q} \|_{\mathrm{TV}} \nonumber\\
&\stepa{\le} 2\| P_{Z^L|Z_0=i, Q=\operatorname{C}_k} - P_{Z^L|Z_0=i, Q} \|_{\mathrm{TV}}\nonumber \\
&\stepb{=} 2\left(1 - \prod_{t=1}^L Q_{i+t-1,i+t}\right)\nonumber \\
&\le 2\left(1 - (1-k^{-c_2})^L\right) \le 2Lk^{-c_2},  \label{eq: perturb analysis}
\end{align}
where (a) follows from the data processing equality for the TV distance and that $P_{X^L|Z^L}$ does not depend on $Q$, and (b) observes that $P_{Z^L|Z_0=i, Q=\operatorname{C}_k}$ is supported on a single path $z^L = (i+1,\cdots,i+L-1)$ (all indices are understood modulo $k$), and $\|P-Q\|_{\mathrm{TV}} = \sum_{x: P(x)>Q(x)}(P(x) - Q(x))$. Consequently, 
\begin{equation}\label{eq: frob ub moment}
    \|E^\star - E\|_{\mathrm{F}}^2 = \sum_{i=1}^k \|e_i^\star - e_i\|_2^2 \le \sum_{i=1}^k \|e_i^\star - e_i\|_1^2 = \widetilde{O}(k^{1-2c_2}), 
\end{equation}
so Mirsky's theorem (cf. Lemma \ref{lem: Mirsky}) shows that for a large enough constant $c_2>0$, the condition number of $E$ is close to the condition number of $E^\star$. This shows that $\kappa = \operatorname{poly}(k)$ with probability at least $1-2k^{-d_6}$ for the first condition in Lemma \ref{lem: tensor decomp}. 

Finally we check the second condition in Lemma \ref{lem: tensor decomp}. Since $\theta_1,\cdots,\theta_k$ are i.i.d.~and uniformly distributed on $[a_1, a_2]$, it holds that
\begin{align*}
\P(|\theta_1 - \theta_2|\le k^{-3}) \lesssim k^{-3}. 
\end{align*}
By a union bound, with probability at least $1-1/k$, we have $|\theta_i - \theta_j|\ge 1/k^3$ for all $i \neq j$. By the definition of $p_i$, this implies that for $i\neq j$, 
\begin{align*}
\left\| \frac{(p_i, 1-p_i)}{\sqrt{p_i^2 + (1-p_i)^2}} -  \frac{(p_j, 1-p_j)}{\sqrt{p_j^2 + (1-p_j)^2}} \right\|_2 \gtrsim |p_i - p_j| \gtrsim |\theta_i - \theta_j| \ge k^{-3}, 
\end{align*}
so that $\delta = \Omega(1/k^3)$ in Lemma \ref{lem: tensor decomp} with probability at least $1-1/k$. 

In summary, for $k\ge c_0$ with a large enough constant $c_0>0$, all conditions in Lemma \ref{lem: tensor decomp} are satisfied for $\kappa, 1/\delta = \operatorname{poly}(k)$ with probability at least $0.99$. By the arguments under Lemma \ref{lem: tensor decomp} we arrive at \eqref{eq:second_step}, as desired. 

\subsubsection{Third step: proof of Lemma \ref{prop: redundancy lower bound}}
Given accurate estimates of $\widehat E_n, \widehat F_n$, we now seek to recover the transition matrix $\widehat Q_n$ with a small error. To this end, we note the following lemma:

     \begin{lemma}[Tensor to Transition, Theorem 4 in \cite{huang2015minimal}]\label{lem: tensor to trans}
    Given matrix $E \in \mathbb{R}^{2^L \times k}$ such that $E_{x^L, z_0}=\mathbb{P}\left(X^L=x^L | Z_0=z_0\right)$ is the conditional forward moment. We marginalize the conditional distribution to $E^{(L-1)}_{x^{L-1}, z_0}=\mathbb{P}\left(X^{L-1}=x^{L-1} | Z_0=z_0\right) \in \mathbb{R}^{2^{L-1} \times k}$. If $E$ has full column rank $k$, then the transition matrix (let $O\in\mathbb{R}^{2\times k}$ be the emission):
$$
Q=\left(O \odot E^{(L-1)}\right)^{\dagger} E\triangleq B^\dagger E.
$$
Specifically, the Khatri-Rao product $O \odot E^{(L-1)}\in \mathbb{R}^{2^L\times k}$ is exactly:
\begin{equation}\label{eq:B_def}
    B_{x_0^{L-1}, z_0}=\left(O \odot E^{(L-1)}\right)_{x_0^{L-1}, z_0}=\P(X_0^{L-1}=x_0^{n-1}|Z_0=z_0)
\end{equation}
and $X^{\dagger}\text { denotes the pseudo-inverse of a matrix } X \text {. }$
 \end{lemma}
Firstly, we show that $\kappa(B)\geq k^{-O(1)}$ whenever the emission is such that Lemma~\ref{lem:singular_value_LB} is satisfied and that the perturbation from $Q$ to $\T_k$ is not large. This is because when $Q=\T_k$, the corresponding $B^\star$ is exactly a column permutation of $E^\star$. In other words, for all emission matrices $O$ one has that $\kappa(E^\star)=\kappa(B^\star)$.
Following the exact same lines as \eqref{eq: perturb analysis} (replacing $Z_0$ with $Z_1$), we get that:
$$\|B-B^\star\|_{\mathrm{F}}=\widetilde O(k^{1-2c_2})$$and hence by Mirsky's theorem we know that the condition number for $\kappa(B)=\operatorname{poly}(k)$ with probability at least $1-2k^{-d_6}$.

Taking the union of such events, we now have $\kappa(B), \kappa(E)$ all upper bounded by $\operatorname{poly}(k)$, and we wish to show that for $\widehat{Q}_n = \widehat{B}_n^\dagger \widehat{E}_n$, the error
$$Q-\widehat Q_n=B^\dagger(E-\widehat E_n)+(B^\dagger - \widehat B_n^\dagger)\widehat{E}_n$$
has small norm. Note that
\begin{align*}
    \|Q-\widehat Q_n\|_{\mathrm{op}}&\leq \|B^\dagger(E-\widehat E_n)\|_{\mathrm{op}}+\|(\widehat B_n^\dagger-B^\dagger)\widehat E_n\|_{\mathrm{op}}\\
    &\leq \|B^\dagger\|_{\mathrm{F}}\|E-\widehat E_n\|_{\mathrm{op}}+\|\widehat E_n\|_{\mathrm{F}}\|\widehat B_n^\dagger-B^\dagger\|_{\mathrm{op}}\\
    &\leq 2^Lk\left(\|E-\widehat E_n\|_{\mathrm{op}}+\|\widehat B_n^\dagger-B^\dagger\|_{\mathrm{op}}\right).
\end{align*}
By \eqref{eq:second_step}, one only need to upper bound $\|\widehat B_n^\dagger-B^\dagger\|_{\mathrm{op}}\in O(n^{-c})$ for some constant $c>0$. Note that
$$\|B-\widehat B_n\|_{\mathrm{F}}\leq \|E^{(L-1)}-\widehat E^{(L-1)}_n\|_{\mathrm{F}}\leq 2\|E-\widehat E_n\|_{\mathrm{F}}$$
by \eqref{eq:B_def} and therefore $\|\widehat B_n-B\|_{\mathrm{F}}\in O(n^{-c'})$ for some $c'>0$ from \eqref{eq:second_step}. Our conclusion follows from the following lemma:
\begin{lemma}[Theorem 4.3 in \cite{stewart1969continuity}]\label{lem: inv cont}
Let $A$ be an $m \times n$ matrix of rank $n$, and let the error matrix be $E$.
Let $\kappa=\|A\|_{\mathrm{op}}\left\|A^{\dagger}\right\|_{\mathrm{op}}$ and $H=(A+E)^{\dagger}-A^{\dagger}$. If
$
\left\|A^{\dagger}\right\|_{\mathrm{op}}\left\|E\right\|_{\mathrm{op}}<1,
$
then
$$
\frac{\|H\|_{\mathrm{op}}}{\left\|A^{\dagger}\right\|_{\mathrm{op}}}<(1+\gamma)\beta,
$$
where
$
\gamma=\left(1-\frac{\kappa\left\|E\right\|_{\mathrm{op}}}{\|A\|_{\mathrm{op}}}\right)^{-1}
$
and
$
\beta=\frac{\gamma \kappa\left\|E\right\|_{\mathrm{op}}}{\|A\|_{\mathrm{op}}}.
$
 \end{lemma}
Plugging in $H=\widehat B_n^\dagger-B^\dagger$ into the above lemma and using the fact that $\kappa(B)$ is upper bounded by a polynomial of $k$, whereas $\|\widehat B_n-B\|_{\mathrm{op}}$ is upper bounded by $n^{-c'}$ for some $c'>0$, we are done.
\subsection{Proof of \eqref{eq: high ent}}
As a final step, we present the following lemma on a high-entropy construction discussed in the main text that guarantees estimation \eqref{eq: lem 4.1 est} indeed leads to redundancy lower bounds:
\begin{lemma}[Distribution on hidden states]\label{lem: high ent construction}
     For any constant $c>0$, there exists a distribution $\mu$ supported on the set $\{Q\in\mathbb{R}^{k\times k}: Q \text{ is double stochastic}, \|Q-\T_k\|_{\max}\leq k^{-c}, Q_{i, i}=0\}$ such  that $h(\mu)\gtrsim -k^2\log k$.
 \end{lemma}
 \begin{proof}
     Consider any pair $(i, j)$ such that $0<i<j-2<k.$ Consider the associated grids $\{(i, j), (j-1, i+1), (i, i+1), (j-1, j)\}$. One can associate an independent random variable $X_{i,j}$ such that $Q_{i,j}=X_{i,j}=Q_{j-1, i+1}$ and $Q_{i, i+1}=1-\sum_j X_{i,j}-\sum_j X_{j,i+1}$, $Q_{j-1, j}=1-\sum_i X_{i,j}-\sum_i X_{j-1,i}$. This construction will always ensure that the resulting matrix is double stochastic (since each $X$ modifies a $2\times 2$ grid). Furthermore, restricting $0\leq X_{i,j}<k^{-c-2}$ ensures that the max offset from the default $\operatorname{C}_k$ (corresponding to all $X=0$) is at most $k^{-c}$ as desired.

     Finally, the entropy is guaranteed as we recover from $Q$ exactly $\Theta(k^2)$ independent random variables that each has a range of $k^{-O(1)}$, when $k>5$. The cases for small $k$ can be verified easily as when $k=3$ there exists a trivial construction with constant entropy. 
 \end{proof}
\section{Computationally predicting HMMs}\label{appdx: comp HMM}

\subsection{Algorithmic upper bound: small $k, \l$}\label{appdx: small k l}
\begin{figure}[t]
    \centering
    \begin{mdframed}[innerrightmargin=15pt]
        \textbf{Algorithm:} Count the number of $z^K\in[k]^K$ with given a transition and emission counts.

        \vspace{5.5pt}
        \textbf{Input:} Matrices $T\in\mathbb{Z}^{k\times \l}, M\in\mathbb{Z}^{k\times k}$ with non-negative entries, where $\sum T_{ij}=1+\sum M_{rs}=K$.
        Emissions $x^K\in [\l]^K.$
        \begin{enumerate}
            \item If $K=1$ and $T_{zx_1}=1$ for some $z$, output 1 directly.
            \item Check that $\sum_{j\in[k]} M_{ij}=\sum_{j\in[\l]}T_{ij}=N_i$ for all except for exactly one $i_0\in[k]$; otherwise, output 0. In this case, $z_K=i_0$ since it is the only item that shows up differently comparing rows of $M, T$.
            \item For $i\in[k]$ let $M^{(i)}\in\mathbb{R}^{k\times k}, T^{(i)}\in\mathbb{R}^{k\times \l}$ be such that
            \begin{align*}
                M^{(i)}_{ii_0}&=M_{ii_0}-1\\
                T^{(i)}_{i_0x_K}&=T_{i_0x_K}-1
            \end{align*}
            and all other entries matching $M, T$ otherwise. This $i$ represents candidate $z_{K-1}$'s.
        \item Run algorithm on $(M^{(i)}, T^{(i)}; x^{K-1})$ for all $i\in [k]$, and sum the results over $i$ according to \prettyref{eq:recursion}.
        \end{enumerate}
        \textbf{Output:} The number of possible trajectories.
    \end{mdframed}
    \caption{Algorithm for computing 
    $\calA(M,T; x^{K})$, the number of satisfying hidden state sequences.}
    \label{fig: alg 1}
\end{figure}
We show that, for any given matrix $M, T$ one can compute the number of satisfying trajectories $z^n\in [k]^n$ efficiently such that the counts matches exactly $M, T$.

\begin{lemma}[See \Cref{fig: alg 1} and Lemma~\ref{lem: const k l alg main}]\label{lem: const k l alg}
    The proposed algorithm $\cal A$ which runs according to the recursion:
    \begin{equation}
    \mathcal{A}(M, T; x^K)=\sum_{i\in [k]}\mathcal{A}(M^{(i)}, T^{(i)}; x^{K-1})
        \label{eq:recursion}
    \end{equation}
    computes exactly the count of $z^K\in [k]^K$ with the given transition/emission counts in time $K^{O(k\l+k^2)}$.
\end{lemma}
\begin{proof}
    The proof is via induction on $K$, assuming that our computation returns the correct result when $K=1$ (in which case it is straightforward to check the count as either 0 or 1). From matching the number of appearances, $z_K=i_0$ for any trajectory with $(M, T)$ in \Cref{fig: alg 1}. One thus sums all trajectories with $(z_{K-1}, z_K)=(i, i_0)$, which is a trajectory counting problem on $z^{K-1}$ corresponding to $(M^{(i)}, T^{(i)})$. Assuming that the count is consistent for $K-1$, the count on $K$ should be consistent as well.

    In terms of the runtime: one can simply create an empty array of size $K^{k\l+k^2}$ first and fill in an item (count) corresponding to some $(M, T)$ at each time some trajectory count is computed. The cost of filling a new item assuming $O(1)$ access 
    to the grid memory
    is at most $O(k)$, and thus the compute filling the entire grid is at most $K^{O(k\l+k^2)}$ assuming $K\in\Omega(k+\l)$. This concludes the runtime.
\end{proof}
Finally, given $Q_{X^n, Z^n}(x^n, z^n)=F(M, T)$ one has that:
$$Q_{X^n}(x^n)=\sum_{z^n} Q_{X^n, Z^n}(x^n, z^n)=\sum_{M, T}F(M, T)\cdot\mathcal{A}(M, T; x^n)$$can be computed in $n^{O(k^2+k\l)}$-time.

\subsection{Algorithmic upper bound: Markov approximation}\label{appdx: any k l}
When $k,\ell$ are moderately large, the above algorithm via marginalization becomes intractable, and we need efficient choices of $Q_{X^{n+1}}$ to achieve a small redundancy in Proposition \ref{prop:riskred}. The idea is to drop the structure in $Z^{n+1}$ (hence no marginalization) and apply a Markov approximation directly to $X^{n+1}$. Specifically, \cite[Lemma 23]{HJW23} shows the existence of $Q_{X^{n+1}}$ that
\begin{align*}
\max_{x^{n+1}} \log \frac{P_{X^d}(x^d)\prod_{t=d+1}^{n+1} P_{X_t|X_{t-d}^{t-1}}(x_t|x_{t-d}^{t-1})}{Q_{X^{n+1}}(x^{n+1})} \lesssim \ell^{d+1}\log \frac{n}{\ell^{d+1}} + d\log \ell
\end{align*}
for $n\ge \ell^{d+1}$, and $Q_{X^{n+1}}$ can be evaluated in time $\operatorname{poly}(n,\ell^{d})$. Taking the expectation over $x^{n+1}\sim P_{X^{n+1}}$ leads to the redundancy upper bound of $Q_{X^{n+1}}$: 
\begin{align*}
\Red(Q_{X^{n+1}}; \PHMM_{n}(k,\ell)) \lesssim \sum_{t=d+1}^{n+1} I(X_t; X^{t-d-1}|X_{t-d}^{t-1}) + \ell^{d+1}\log \frac{n}{\ell^{d+1}} + d\log \ell,
\end{align*}
where the first term is further upper bounded by
\begin{align*}
(n+1)I(X_{n+1}; X^{n-d} | X_{n+1-d}^n) \stepa{\le} \frac{n+1}{d+1}\sum_{t=0}^n I(X_{n+1}; X^{n-t} | X_{n-t+1}^n)
\stepb{\le} \frac{n+1}{d+1}\log k. 
\end{align*}
Here (a) is because of the decreasing property of $t\mapsto I(X_{n+1}; X^{n-t} | X_{n-t+1}^n)$, and (b) follows from Proposition \ref{prop: decay info}. Consequently, by Propositions \ref{prop:riskred} and \ref{prop: decay info}, this choice of $Q_{X^{n+1}}$ leads to the prediction risk
\begin{align*}
\Risk(\widetilde{Q}_{X_{n+1}|X^n}, \PHMM_n(k,\ell)) \lesssim \frac{\log k}{d} + \frac{\ell^{d+1}}{n}\log \frac{n}{\ell^{d+1}} + \frac{d\log \ell + \log k}{n}. 
\end{align*}
Choosing $d = \log n/(2\log \ell)$ leads to the risk upper bound in Theorem \ref{thm:computation_UB}. The overall computational time is $\operatorname{poly}(n,\ell^d) = \operatorname{poly}(n)$, as desired.

\subsection{Computational lower bounds}\label{appdx: comp lower}
In the last part of our computational discussions we sketch two lower bounds, in contrast with our $O(\frac{\log k\log \l}{\log n})$ upper bound. Our lower bounds will be based on cryptographic assumptions involving the {Learning Parity with Noise} (LPN) problem (\cite{blum2003noise}) and {refutation of a class of Constraints Satisfying Problem (CSP) (\cite{feldman2015complexity}). In particular, the following assumptions are observed.
 \begin{conjecture}[Learning Parity With Noise, see e.g. \cite{Wiggers21practical}]\label{conj: lpn}
 Let the secret key $\mathbf{s}\sim_{\operatorname{unif}}\mathbb{F}_2^k$ and noise $\eta=0.05$. Any polynomial-time algorithm on the Learning Parity with Noise (LPN) problem with $\Omega(1)$-time query access to a noisy observation $y=\langle \mathbf{s}, \mathbf{x} \rangle\oplus\Bern(\eta)$ for $\mathbf{x}\sim_{\operatorname{unif}}\mathbb{F}_2^k$,   requires $2^{\Omega(k/\log k)}$ computational complexity to decide between pure noise ($y\sim\Bern(1/2)$) and noisy parity ($y=\langle \mathbf{s}, \mathbf{x} \rangle\oplus\Bern(\eta)$) correctly with probability $2/3$.
 \end{conjecture}
 \begin{conjecture}[Refuting CSP's, see e.g. \cite{kothari2017sum}]\label{conj: csp}
    Let $k>r$ be constants,  $Q$ be any distribution over $k$-clauses with $N$ variables of complexity $r$ and $0<\eta<1$. Any polynomial-time algorithm that, given access to a distribution $D$ that equals either the uniform distribution over $k$-clauses $U_k$ or a (noisy) planted distribution $Q_{\boldsymbol{\sigma}}^\eta=(1-\eta) Q_{\boldsymbol{\sigma}}+\eta U_k$ for some $\boldsymbol{\sigma} \in\{0,1\}^n$ and planted distribution $Q_\sigma$, decides correctly whether $D=Q_\sigma^\eta$ or $D=U_k$ with probability $2 / 3$ needs $\widetilde{\Omega}\left(N^{r / 2}\right)$ clauses. (Here $\widetilde\Omega$ ignores $\log$ factors.)
 \end{conjecture}
 Given the above assumptions, our lower bounds are as follows:
 \begin{theorem}[Computational lower bounds for HMM prediction]\label{thm: hmm comp lower}
 The following holds:
     \begin{enumerate}
         \item For any $\eps>0$, if $k\geq \log^{1+\eps} n$ and $\l\geq 2$, then there exists a distribution on HMMs where no efficient algorithm can achieve $o(\frac{\log k}{\log n\log\log n})$ error, assuming Conjecture~\ref{conj: lpn}.
         \item For every $\alpha>0$ there exists $k_\alpha\geq 2$, such that if $k\geq k_\alpha$ and $\l\geq n^\alpha$, there exists a distribution on HMMs where no efficient algorithm can achieve $o(1)$ error assuming Conjecture~\ref{conj: csp}.
     \end{enumerate}
 \end{theorem}
\begin{remark}
    This result, combined with \Cref{thm:computation_UB}, leaves the following cases of interest open in terms of computational algorithms:
    \begin{enumerate}
        \item For $k=n^{\Omega(1)}$ and $\l=2$, can there be efficient algorithm achieving $o(1)$ risk? 
        \item For $k=O(1)$ and $\l=\operatorname{polylog}(n)$, can there be computational lower bounds of $1/\operatorname{polylog}(n)$ for prediction risk?
    \end{enumerate}
\end{remark}
The embedding of cryptographically hard models into computational lower bounds in HMM has been long observed in various prior literature (e.g. \cite{mossel2005learning,sharan2018prediction}). Here we adopt these constructions into our setting.
\begin{proof}[Proof of \Cref{thm: hmm comp lower}] We divide the two cases:
    \begin{enumerate}
    \item For the $\l=2$ case. Let $s=\lfloor\log_2 \frac{k}{r}\rfloor-2$ for some $r$ and let the $k$ hidden states be labeled:
    $$Z=\{(i, b_0, b_1, b_2,\dots b_s); b\in \{0, 1\}, i=1,2,\dots, r+s\}$$
    and hence $|Z|=(r+s)2^{r+1}\leq r2^{r+2}\leq k$. We will choose $r=C\log n\log\log n$ for a large constant $C$, so that $s \in \Omega(\log k)$ thanks to the assumption $k \ge \log^{1+\varepsilon} n$. 

    Let $(s_{i,j})_{i\in [r], j\in [s]}$ be independent $\Bern(1/2)$ secret keys, so that there are $s$ secret keys in total, each of length $r$. The transitions and emissions are defined as follows: 
    \begin{itemize}
        \item Emission: state $(i, b_0^s)$ emits $b_0$ if $i\le r$, and $b_{i-r} \oplus \Bern(\eta)$ if $i\in \{r+1,\cdots,r+s\}$; 
        \item Transition: state $(i, b_0^s)$ goes to $(i+1, c_0^s)$ (as usual, $r+s$ goes to $1$), where:
        \begin{enumerate}
            \item If $i\in \{r,r+1,\cdots,r+s-1\}$, let $c_0^s = b_0^s$; 
            \item If $i\in \{r+s,1,\cdots,r-1\}$, sample $c_0 \sim \Bern(1/2)$, and let $c_j = b_j \oplus s_{i+1,j}c_0$ for all $j=1,\cdots,s$ (as usual, $s_{i+1,j} = s_{1,j}$ when $i=r+s$).
        \end{enumerate}
    \end{itemize}
    In other words, the transition runs in cycles of length $r+s$. During the first $r$ rounds in each cycle, the learner observes $\mathbf{b}_0=(b_{0,1},\cdots,b_{0,r})\sim \mathsf{Unif}(\mathbb{F}_2^r)$. For the $(r+j)$-th round with $j\in [s]$, under the current transition the learner observes
    \begin{align*}
    \langle \mathbf{b}_0, \mathbf{s}_j \rangle \oplus \Bern(\eta), \quad \text{where } \mathbf{s}_j = (s_{1,j},\cdots,s_{r,j}). 
    \end{align*}
    In other words, each cycle consists of one query to each of $s$ independent LPN instances. Since there are $\le n$ cycles in total, each LPN instance has sample size at most $n$. 

    Next we understand the prediction problem of the current HMM. Clearly, under the stationary distribution we have $i\sim \mathsf{Unif}([r+s])$, so that $\P(i\in \{r+1,\cdots,r+s\})= s/(r+s)$ for the state at time $n+1$. Again, using Lemma \ref{lem: partial_reveal} and \eqref{eq: decay info}, we assume without loss of generality that the starting state $i$ at $t=0$ is known to the learner, so that the learner knows the relative location of time $n+1$ in a given cycle. Suppose time $n+1$ is the $(r+j)$-th round of some cycle, then predicting $X_{n+1}$ is the same as predicting the distribution $\langle \mathbf{b}_0, \mathbf{s}_j \rangle \oplus \Bern(\eta)$ with observed $\mathbf{b}_0=(b_{0,1},\cdots,b_{0,r})\sim \mathsf{Unif}(\mathbb{F}_2^r)$ and a hidden secret $\mathbf{s}_j$. As $n \le 2^{cr/\log r}$ with $c\to 0$ as $C\to\infty$, Conjecture \ref{conj: lpn} implies that the KL prediction error is $\Omega(1)$ by choosing $C$ large enough. Consequently, the overall KL prediction risk is lower bounded by
    \begin{align*}
        \Omega(1)\cdot \frac{s}{r+s} = \Omega\left(\frac{\log k}{\log n\log\log n}\right). 
    \end{align*}
    Now choosing the growth of $\omega$ in the definition of $r$ arbitrarily slow gives the claim. 
    
    \item For the $\l=n^\alpha$ case, this follows directly from plugging in the parameter correspondence to Theorem 2 in \cite{sharan2018prediction} while leaking the first hidden state in the fashion of Lemma~\ref{lem: partial_reveal}\footnote{The results in \cite{sharan2018prediction} does not require this lemma as they assumed a slightly different loss; see \Cref{appdx: losses}. Here we bypass this issue by leaking $U=Z_1$ to both $P^{\operatorname{HMM}}$ and $Q$ and adjust the result via Lemma~\ref{lem: partial_reveal}.}. In short, when $r$ is a large enough constant in Conjecture~\ref{conj: csp}, there exists a distribution on constant-size clauses such that detection is impossible on the CSP problem which can be embedded with $O(1)$ hidden states. Therefore, with constant probability, one cannot distinguish the next bit from random.
    \end{enumerate}
\end{proof}

\section{Lower bound proof for renewal processes}\label{appdx: renewal}

In this section we prove the lower bound part of \prettyref{thm: risk renewal}, namely, $\Riskrenew(n) \gtrsim \sqrt{n^{-1}}$.
Similar to the strategies proving lower bounds in \Cref{appdx: lower HMM}}, we consider a Bayesian setting where the model parameter (in this case the interarrival distribution $\mu$) is random and drawn from some prior. Then the Bayes prediction KL risk is given by the  conditional mutual information 
$I(\mu;X_{n+1}|X^n)$, which we aim to show is at least $\Omega(\sqrt{n^{-1}})$.

Let us first recall the equivalent HMM representation for renewal processes from \prettyref{sec:ub-renewal} with state space $\naturals$. The stationary distribution on the hidden states is the same as the distribution of the initial wait time $T_0$, given by:
    \begin{equation}
       \pi_\mu(i) = \frac{1}{m(\mu)}\sum_{j> i}\mu(j)
    \label{eq:stationary-renewal}
    \end{equation}
    where $m(\mu) = \sum_{i\geq 1} i \mu(i)$ is the mean of $\mu$. 
Notably, $X^n$ has the same law as its time reversal. 
This reversibility will be exploited in our proof of the lower bound. 
    

We will consider a prior under which $\mu$ is always finitely supported. Let the last appearance of ``1'' in $X^n$ be $X_K$ for $K\leq n$, and let $T\triangleq\one_{n+1-K\in\operatorname{supp}(\mu)}$. Note that $T=0$ implies $X_{n+1}=0$ almost surely. Denote by $p(X^{n}, \mu)\triangleq \P(X_{n+1}=1|X^n, \mu)$ the optimal predictor who knows the model parameter $\mu$. By \prettyref{eq:hazard}, this is given by the hazard ratio of $\mu$, namely
\begin{equation}
p(X^{n}, \mu)
=\frac{\mu(n+1-K)}{\sum_{d\geq n+1-K} \mu(d)}.
\label{eq:hazard1}
\end{equation}
Without knowing $\mu$, the predictor is the average of $p(X^{n}, \mu)$ over the posterior law of $\mu$ given the data $X^n$. Let $p(X^n)\triangleq\P(X_{n+1}=1|X^n) = \E_{\mu | X^n}[p(X^{n}, \mu)]$. 

Let $E_1$ be some event measurable with respect to $\sigma(X^n)$ to be specified. Let $E_2$ be the event of $T=0$, which is measurable with respect to $\sigma(X^n,\mu)$. For any estimator $Q$, its average risk (with expectations taken over both data $X^n$ and $\mu$ according to the prior) satisfies
    \begin{align}
    \E[\kl(P_{X_{n+1}| X^n, \mu} \| Q_{X_{n+1}|X^n})]
    &\ge \P(E_1)\cdot \E[\kl(P_{X_{n+1}|X^n, \mu} \| Q_{X_{n+1}|X^n}) | E_1]\nonumber \\
    &\ge \P(E_1)\cdot \E[\kl(P_{X_{n+1} | X^n, \mu} \| P_{X_{n+1} | X^n}) | E_1]\nonumber \\
    &= \P(E_1)\cdot \E[\kl( \Bern(p(X^n, \mu)) \| \Bern(p(X^n)) ) | E_1]\nonumber \\&\ge \P(E_1\cap E_2)\cdot \E[\kl\left( \Bern(0) \| \Bern(p(X^n))\right) | E_1\cap E_2]\nonumber \\
    &\ge \P(E_1\cap E_2)\cdot \E[p(X^n)|E_1\cap E_2] \label{eq: risk rnwl E12} 
    \end{align}
    where we used the fact that $T=0$ implies $p(X^n, \mu)=0$. Furthermore:
    \begin{align}\label{eq:risk rnwl p}
        p(X^n) = \E_{\mu | X^n}[p(X^n,\mu)]=\P(E_2^c|X^n)\E_{\mu | X^n, E_2^c}[p(X^n,\mu)].
    \end{align}
    We will choose a prior under which $\mu$ is always uniform over $\Theta(\sqrt{n})$ integers. Therefore, by \prettyref{eq:hazard1}, 
    for all $(\mu, X^n)\in E_2^c$, $p(X^n,\mu)\gtrsim \sqrt{n^{-1}}$. Furthermore, suppose we can show that there exists an event $E_1$ and a constant $c>0$ such that: (a) $\P(E_1)>c$ and (b) for all $X^n\in E_1$, one has $\P(T=0|X^n)=\P(E_2|X^n)\in(c, 1-c)$.
   Then the last line of \eqref{eq: risk rnwl E12} is lower bounded by the desired $\Omega(n^{-1/2})$ rate. In the following, we show the construction and proof.

We consider a prior that was previously used for proving the redundancy lower bound in \cite{csiszar96redundancy}. There the goal is to prove that $I(\mu;X^n) \gtrsim \sqrt{n}$ as opposed to $I(\mu;X_{n+1}|X^n) \gtrsim \frac{1}{\sqrt{n}}$ here. Let $a_n=C\sqrt{n}$ be an even number for some large constant $C$. Let the interarrival distribution $\mu$ to be the uniform distribution on an $a_n$-subset of $[2a_n]$, with its support chosen uniformly over all such sets where exactly half of its elements lies in $[0, a_n]$. In this way, for all $(X^n, \mu)\in E_2^c$ one has that $p(X^n, \mu)\geq \mu(n+1-K)\geq \frac{1}{a_n}\asymp\sqrt{n^{-1}}$.

Define $E_1$ to be the intersection of the following events: \begin{enumerate}
        \item[(1)] The last appearance of $X_K=1$ satisfy $K>n-a_n$.
        \item[(2)] There are at most distinct $\sqrt{n}$ interarrival times in $X^n$
        (known as the renewal types \cite{csiszar96redundancy}).
        Denote this set of interarrivial times by $A$.
        \item[(3)] The gap $n+1-K$ has never appeared in the past $\sqrt{n}$ interarrivals backwards from $X_K$.
    \end{enumerate}
Clearly $E_1$ is $\sigma(X^n)$-measurable. We show that $\P(E_1) = \Omega(1)$ for a large enough constant $C$. By the time reversal property of the renewal process, the distribution of $n+1 - K$ is given by $\pi_\mu$ in \eqref{eq:stationary-renewal}. As $m(\mu)\le 2a_n$ and 
\begin{align*}
\sum_{i=1}^{a_n} \sum_{j>i} \mu(j) \ge \sum_{j=a_n+1}^{2a_n} a_n \mu(j) = a_n \mu([a_n+1, 2a_n]) = \frac{a_n}{2}, 
\end{align*}
event (1) happens with probability at least $1/4$. As for (2), the expectation of the interarrival time is $\ge a_n/2 = C\sqrt{n}/2$. Consequently, for $C>3$, Hoeffding's inequality shows that event (2) happens with probability $1-o_n(1)$. For (3), each interarrival time equals to $n+1-K$ with probability at most $1/a_n$. By a union bound, event (3) happens with probability at least $1 - \sqrt{n}/a_n = 1 - 1/C$. A union bound then gives that $\P(E_1)\ge 1/4 - 1/C - o_n(1)$, which is $\Omega(1)$ for large enough $(n,C)$. 

Next we show that $\P(T=0|X^n) \in (c,1-c)$ whenever $E_1$ holds. Let us first prove a simple lemma:
\begin{lemma}\label{lem: renewal lem}
   For any $\mu_1, \mu_2$ supported in the prior and any $k\in [a_0]$, it holds that
   $$\frac{1}{8}\le \frac{\P(K=n+1-k|\mu_1)}{\P(K=n+1-k|\mu_2)}\le 8.$$
\end{lemma}
\begin{proof}
By the time-reversal property of the renewal process, the conditional distribution of $n+1-K$ conditioned on $\mu$ is given by $\pi_\mu$ in \prettyref{eq:stationary-renewal}. Consequently, 
    $$\P(K=n+1-k|\mu)=\frac{1}{m(\mu)}\sum_{j>k}\mu(j).$$
Since the support of $\mu$ has $a_n/2$ elements in $[a_n]$ and $a_n/2$ elements in $[a_n+1,2a_n]$, we have $m(\mu)\in [a_n/2, 2a_n]$. In addition, as $k\in [a_0]$, we have $\sum_{j>k} \mu(j) \in [1/2,1]$. This gives $\P(K=n+1-k|\mu)\in [1/(4a_n), 2/a_n]$ for all $\mu$ in the support of the prior, and the lemma follows. 
\end{proof}
    Now we show that for all $X^n\in E_1$ with $K=K(X^n)=n+1-k$, the ratio
    $$\frac{\P(T=1|X^n)}{\P(T=0|X^n)}=\frac{\P(T=1|X^n, K=n+1-k)}{\P(T=0|X^n, K=n+1-k)}$$ is bounded above and below by positive constants. Since $E_1$ holds, $k\le a_n$. First of all, 
    \begin{align}\label{eq:bayes_1}
    \frac{\P(T=1|K=n+1-k)}{\P(T=0|K=n+1-k)} &= \frac{\P(k\in \supp(\mu) | K=n+1-k)}{\P(k\notin \supp(\mu) | K=n+1-k)} \nonumber\\
    &= \frac{\P(k\in \supp(\mu))}{\P(k\notin \supp(\mu))}\frac{\P(K=n+1-k| k\in \supp(\mu))}{\P(K=n+1-k|k\notin \supp(\mu))} \nonumber\\
    &= \frac{\P(K=n+1-k| k\in \supp(\mu))}{\P(K=n+1-k|k\notin \supp(\mu))} =\Theta(1), 
    \end{align}
    where the last step is due to Lemma~\ref{lem: renewal lem}. Second, for the set $A=A(X^n)$ consisting of distinct interarrival times in $X^n$, the event $E_1$ implies that $k\notin A$ and $|A|\le \sqrt{n}$. Therefore, 
    \begin{align}\label{eq:bayes_2}
    &\frac{\P(A\subseteq\supp(\mu)|K=n+1-k, k\in\supp(\mu))}{\P(A\subseteq\supp(\mu)|K=n+1-k, k\not\in\supp(\mu))} \nonumber\\
    &= \frac{\P(K=n+1-k|k\notin \supp(\mu)) }{\P(K=n+1-k|k\in \supp(\mu))}\frac{\P(K=n+1-k|k\in \supp(\mu),A\subseteq \supp(\mu))}{\P(K=n+1-k|k\notin \supp(\mu),A\subseteq \supp(\mu))}\nonumber\\
    &\qquad \cdot \frac{\P(A\subseteq \supp(\mu),k\in \supp(\mu))}{\P(A\subseteq \supp(\mu),k\notin \supp(\mu))}\frac{\P(k\notin \supp(\mu))}{\P(k\in \supp(\mu))}\nonumber\\
    &\stepa{=} \Theta(1)\cdot \frac{\binom{a_n-|A\cap [a_n]|-1}{a_n/2}}{\binom{a_n-|A\cap [a_n]|-1}{a_n/2-1}} = \Theta(1)\cdot \frac{a_n/2 - |A\cap [a_n]|}{a_n/2} \stepb{=} \Theta(1), 
    \end{align}
    where (a) is due to Lemma~\ref{lem: renewal lem}, and (b) uses $|A\cap [a_n]|\le |A|\le \sqrt{n} \le a_n/3$ as long as $C\ge 3$. Finally, by writing down the joint pmf of $X^n$ after time reversal, it is clear that
    \begin{align}\label{eq:bayes_3}
        \frac{\P(X^n|K=n+1-k, T=1, A\subseteq\supp(\mu))}{\P(X^n|K=n+1-k, T=0, A\subseteq\supp(\mu))} = 1.
    \end{align}

    Combining the above results leads to
    \begin{align*}
        \frac{\P(T=1|X^n)}{\P(T=0|X^n)} &=\frac{\P(T=1|X^n, K=n+1-k)}{\P(T=0|X^n, K=n+1-k)} \\
        &= \frac{\P(T=1|K=n+1-k)}{\P(T=0|K=n+1-k)} \cdot \frac{\P(X^n|T=1,K=n+1-k)}{\P(X^n|T=0,K=n+1-k)} \\
        &\overset{\prettyref{eq:bayes_1}}{=} \Theta(1)\cdot \frac{\P(X^n|T=1,K=n+1-k)}{\P(X^n|T=0,K=n+1-k)} \\
        &\overset{\prettyref{eq:bayes_3}}{=}  \Theta(1)\cdot \frac{\P(A\subseteq \supp(\mu)|T=1,K=n+1-k)}{\P(A\subseteq \supp(\mu)|T=0,K=n+1-k)} \overset{\prettyref{eq:bayes_2}}{=} \Theta(1).
    \end{align*}
Therefore, both conditions $\P(E_1) \ge c$ and $\P(T=0|X^n)\in (c,1-c)$ for $X^n\in E_1$ are established, and the $\Omega(n^{-1/2})$ lower bound follows from \eqref{eq: risk rnwl E12} and \eqref{eq:risk rnwl p}.

\end{document}